\documentclass[pdflatex,sn-nature]{sn-jnl}

\usepackage{graphicx}%
\usepackage{multirow}%
\usepackage{amsmath,amssymb,amsfonts}%
\usepackage{amsthm}%
\usepackage{mathrsfs}%
\usepackage[title]{appendix}%
\usepackage{xcolor}%
\usepackage{textcomp}%
\usepackage{manyfoot}%
\usepackage{booktabs}%
\usepackage{algorithm}%
\usepackage{algorithmicx}%
\usepackage{algpseudocode}%
\usepackage{listings}%
\newtheorem{theo}{Theorem}[section]
\newtheorem{cor}{Corollary}[theo]
\newtheorem{lem}{Lemma}[section]

\theoremstyle{definition}

\newtheorem{ex}{Example}[section]

\newtheorem{ass}{Assumption}[section]

\numberwithin{equation}{section}

\newcommand{\rb}[1]{ \left( #1 \right) }
\newcommand{\rrb}[1]{ \left\lbrace #1 \right\rbrace }
\newcommand{\curbr}[1]{ \lbrace #1 \rbrace }

\newcommand{\abs}[1]{\left|#1\right|}

\newcommand{\suml}{\sum\limits}

\newcommand{\intl}{\int\limits}
\newcommand{\ointl}{\oint\limits}
\newcommand{\liml}{\lim\limits}
\newcommand{\supl}{\sup\limits}
\newcommand{\infl}{\inf\limits}

\newcommand{\Res}[2]{\underset{#1}{\operatorname{Res}~}{#2}}
\newcommand{\LandauO}{ \mathcal{O} }
\newcommand{\bigO}[1]{ \mathcal{O}\left\lbrace #1 \right\rbrace }

\newcommand{\C}{\mathbb{C}}

\newcommand{\N}{\mathbb{N}}

\newcommand{\R}{\mathbb{R}}

\newcommand{\Z}{\mathbb{Z}}

\newcommand{\Beta}{\operatorname{B}}

\newcommand{\Epsilon}{\operatorname{E}}

\newcommand{\Ypsilon}{\operatorname{Y}}

\theoremstyle{thmstyleone}%

\theoremstyle{thmstyletwo}%

\theoremstyle{thmstylethree}%

\begin{document}

\title[Generalizing Laplace's method by means of Mellin transforms]{Generalizing Laplace's method by means of Mellin transforms}

\author*[1]{\fnm{Henrik} \sur{Kaiser}}\email{hkaiser@icloud.com}

\affil*[1]{\orgaddress{\city{Schotten}, \postcode{63679}, \state{Hessia}, \country{Germany}}}

\abstract{A well-known procedure for the asymptotic evaluation of Laplace transforms is Laplace's method. Despite its wide applicability, however, it is easy to find relevant examples where the technique is infeasible, because the integrand admits no power series expansion near the critical point, e.g., due to exponential behaviour there. We circumvent this issue through an extension of the method of Mellin transforms, based on an integral representation for the kernel and a generalization of the Taylor expansion. Our main result then extends the Laplacian method in the sense that it provides asymptotic expansions for a wider scope of Laplace transforms in terms of known special functions, rather than only in powers of the asymptotic parameter. The results are illustrated in two examples.}



\maketitle

\section{Introduction}

For $\lambda \geq 0$, an open $\mathcal{P} \subset \R$, as well as continuous functions $\phi : \mathcal{P} \rightarrow (0, \infty)$ and $a : \mathcal{P} \rightarrow \C$, consider the generalized Laplace transform
\begin{align} \label{AsymptExpByAnalyticCont1}
\mathcal{L}_{ \phi, \mathcal{P} }\curbr{a}(\lambda) := \intl_\mathcal{P} e^{-\lambda\phi(t)} a(t) dt.
\end{align}
We concisely summarize the phase $\phi$ and the amplitude $a$ as the ingredient functions of $\mathcal{L}_{ \phi, \mathcal{P} }\curbr{a}(\lambda)$. Supposing absolute convergence of $\mathcal{L}_{ \phi, \mathcal{P} }\curbr{a}(0)$, the monotonicity of the kernel implies absolute convergence of $\mathcal{L}_{ \phi, \mathcal{P} }\curbr{a}(\lambda)$, for all $\lambda \geq 0$, and even decay, as $\lambda \rightarrow \infty$. The integral is specifically of Laplace-type, i.e., the main contribution to the rate for large $\lambda$ comes from a neighborhood of the so-called critical points, the points where the phase approaches its infimum value. The rate of decay is of exponential order, if $\inf_{t \in \mathcal{P}} \phi(t) > 0$, whereas the effect of a zero infimum substantially depends on the local behaviour of the ingredient functions near the critical points. Laplace-type integrals play an important role in mathematical and physical applications, as they constitute a widely admissible standard form, well-suited for an asymptotic ana\-lysis. The most common technique for their asymptotic evaluation is Laplace's method, with Watson's lemma as a special case and the method of steepest descent as a sophisticated extension (see \cite{bleisteinhandelsman1986, Olver1974, paris_2011}). Yet, these procedures require the existence of power series expansions, at least formally, for phase and amplitude function in a neighborhood of the critical points. On the other hand, one can easily construct examples, for which this does not hold, e.g., consider $\mathcal{L}_{ \phi_1, \mathcal{P}_1 }\curbr{a_1}(\lambda)$ with $\mathcal{P}_1 := (0, 1)$, $\phi_1(t) := -\log(1-\exp\curbr{-t^{-1}})$ and $a_1(t) := t$. In order to be able to deal with such si\-tuations, we here propose a generalization of the Laplacian method that is not confined to algebraic ingredient functions but also covers logarithmic and some kind of exponential behaviour. With $\varphi : \mathcal{P} \rightarrow (0, \infty)$ and branch of the argument $\arg\rrb{\varphi(t)} = 0$, for $t \in \mathcal{P}$, define
\begin{align}  \label{AsymptExpByAnalyticCont10}
\mathcal{M}_{\varphi, \mathcal{P} }\curbr{a}(z) := \intl_\mathcal{P} \rrb{\varphi(t)}^{-z} a(t) dt.
\end{align}
Then, being an extension of the well-known method of Mellin transforms (cf. \cite{ParKam2001}), the key idea of our approach is that the asymptotic behaviour of $\mathcal{L}_{ \phi, \mathcal{P} }\curbr{a}(\lambda)$ is connected to the ana\-lytic structure of the generating function $\mathcal{M}_{\varphi, \mathcal{P} }\curbr{a}(z)$ via an appropriate integral representation for the kernel, i.e., $\varphi$ and $\phi$ are related through an integral equation. Clearly, $\mathcal{M}_{\varphi, \mathcal{P} }\curbr{a}(z)$ constitutes a generalized Mellin transform, and its ana\-lytic structure in turn is determined by the local properties of its ingredient functions $\varphi$ and $a$, at the points where $\varphi$ is minimal. These coincide with the critical points of $\phi$, whenever $\varphi$ is a monotonic transformation of $\phi$.
\\
\hspace*{1em}To outline our strategy in a more technical way, we exemplarily consider the Cahen-Mellin integral, i.e., the inverse Mellin transform of the exponential function, according to which, for any $\lambda, x_0 > 0$ and $t \in \mathcal{P}$ with $\phi(t) > 0$, we can write
\begin{align} \label{AsymptExpByAnalyticCont12}
e^{ -\lambda \phi(t) } = \frac{1}{2\pi i} \intl_{x_0-i\infty}^{x_0+i\infty} \rrb{ \phi(t) }^{-z} \lambda^{-z} \Gamma(z) dz.
\end{align}
An application to the integral $\mathcal{L}_{\phi, \mathcal{P} }\curbr{a}(\lambda)$ and a formal interchange in the order of integration yield a representation as a so-called Mellin-Barnes integral, that is
\begin{align*}
\mathcal{L}_{\phi, \mathcal{P} }\curbr{a}(\lambda) = \frac{1}{2\pi i} \intl_{x_0-i\infty}^{x_0+i\infty} \lambda^{-z} \Gamma(z) \mathcal{M}_{\phi, \mathcal{P} }\curbr{a}(z) dz.
\end{align*}
Admissible values of $x_0 > 0$, for which this step is valid, depend on the region of absolute convergence of $\mathcal{M}_{\phi, \mathcal{P} }\curbr{a}(z)$ and thus on the behaviour of $a$ and $\phi$ at the critical points. There, $\curbr{\phi}^{-x_0}$ is maximal and possibly even unbounded. The factor $\lambda^{-z}$ in the integral has asymptotic scaling properties, because its $\lambda$-asymptotic behaviour is characterized through the magnitude of $\Re z$. In other words, $(\lambda^{-z_n})_{n\in\N_0}$ is an asymptotic sequence (or scale) for each $(z_n)_{n\in\N}\subset\C$, with $\Re z_n > \Re z_{n-1}$ (see Ch. 1, $\S$10 in \cite{Olver1974}). As a direct consequence, under absolute convergence, it is easy to see that $\mathcal{L}_{\phi, \mathcal{P} }\curbr{a}(\lambda) = \LandauO(\lambda^{-x_0})$, as $\lambda \rightarrow \infty$, i.e., the real part of the integration path indicates the asymptotic behaviour of the integral. A distinction between three cases can be made. In the first case, it exists $\chi_0>0$, such that absolute convergence of $\mathcal{M}_{\phi, \mathcal{P} }\curbr{a}(z)$ is confined to the strip $0 < \Re z < \chi_0$. The integral is then often also analytic there, and admits an analytic continuation beyond the boundary line $\Re z = \chi_0$. This analytic continuation naturally depends on the functions $\phi$ and $a$ and eventually specifies the asymptotic behaviour of $\mathcal{L}_{\phi, \mathcal{P} }\curbr{a}(\lambda)$. Below, we will see that $\mathcal{M}_{\phi, \mathcal{P} }\curbr{a}(z)$ indeed can be extended to the whole complex plane as a meromorphic function with an equidistant set of simple poles, if the critical point lies on the boundary or in the interior of a domain where $\phi$ and $a$ exhibit analyticity. Similar to the ordinary Mellin transform, it can be unlocked by means of a power series expansion. More precisely, we adopt a result that is known as Bürmann's theorem, a generalization of Taylor's theorem, to expand $\curbr{\phi'}^{-1}a$ in powers of $\phi$. The second possible case is that $\mathcal{M}_{\phi, \mathcal{P} }\curbr{a}(z)$ converges absolutely in the entire half plane $\Re z > 0$, so that $x_0>0$ can be chosen arbitrary, thereby suggesting a rate that exceeds any power of $\lambda$. Finally, there need not exist $\Re z > 0$, such that $\mathcal{M}_{\phi, \mathcal{P} }\curbr{a}(z)$ converges absolutely (e.g., recall the example $\mathcal{L}_{ \phi_1, \mathcal{P}_1 }\curbr{a_1}(\lambda)$ from before). In these circumstances, the above Mellin-Barnes representation is invalid for any $x_0>0$, and the rate of decay of $\mathcal{L}_{\phi, \mathcal{P} }\curbr{a}(\lambda)$ is presumably below any power of $\lambda$, i.e., slower than algebraic. In the present text, we will also show that the second and third situation can be circumvented by choosing appropriate integral representations for the kernel. As a result, instead of ordinary asymptotic power series (also known as Poincaré expansions), we get expansions for $\mathcal{L}_{\phi, \mathcal{P} }\curbr{a}(\lambda)$ in terms of known special functions. For properties of asymptotic expansions and generalizations of this notion, we refer to Ch. I, $\S\S$7 and 10 in \cite{Olver1974} and Ch. I, $\S\S$2 and 3 in \cite{Wong2001}.
\\
\hspace*{1em}The paper is organized as follows. After a brief discussion of prerequisites in $\S$\ref{SecPrel}, we start with a technical study of $\mathcal{M}_{\varphi, \mathcal{P} }\curbr{a}(z)$ in $\S$\ref{AnContTeixExp}. In $\S$\ref{SecDerivAsympStatI}, these results are then applied to the Laplace transform $\mathcal{L}_{\phi, \mathcal{P} }\curbr{a}(\lambda)$ for the derivation of various expansions. A few concrete examples, which actually motivated this contribution, illustrate the capabilities of our approach in $\S$\ref{SecApp}, before the article is concluded by $\S$\ref{SecConc}. Supplementary results can be found in two appendices.

\section{Preliminaries} \label{SecPrel}

We begin with a short overview on some definitions, that will be of frequent use throughout this text, and provide important background information on mapping properties of analytic functions. First of all, unless stated otherwise, the integration path $\mathcal{P}$ is supposed to have the following properties.

\begin{ass} \label{AnContTeixExpA}
The path $\mathcal{P} \subset \R$ is the open segment of the real axis between the fixed points $t_0 \in \R$ and $T \in \R \setminus \curbr{t_0} \cup \curbr{\pm\infty}$.
\end{ass}

So, by construction, $\mathcal{P}$ can be infinite and has endpoints $t_0$ and $T$, of which $t_0$ must be finite. Letting $-\pi < \arg(t-t_0) \leq \pi$ be the principal branch of the argument, we denote by
\begin{align} \label{ZeroAlgBound6}
\omega_\mathcal{P} := \liml_{ \substack{ t \rightarrow t_0 \\ t \in \mathcal{P} } } \arg(t-t_0)
\end{align}
the slope of the integration path at the point $t_0$. Clearly, $\omega_\mathcal{P} \in \rrb{0,\pi}$, since $\mathcal{P} \subset \R$. Actually, $\arg(t-t_0) = \omega_\mathcal{P}$, for all $t \in \mathcal{P}$. For $f : \mathcal{P} \rightarrow \C$ and $\tau \in \curbr{t_0, T}$, we moreover write
\begin{align} \label{DerivAsympStatILAMBDA}
f(\tau) := \liml_{ \substack{t \rightarrow \tau \\ t \in \mathcal{P} } } f(t).
\end{align}
The $k$-th derivative of $f$ is indicated by $f^{(k)}$, for $k \in \N$, or by $k$ primes, for small values of $k$. Following the usual notation, $D_R(t_0) := \curbr{ t \in \C : |t-t_0| < R }$ stands for the open disc of radius $R > 0$ around $t_0$ and $\overline{D}_R(t_0) := \curbr{ t \in \C : |t-t_0| \leq R }$ for its closure. The associated positively (i.e., counterclockwisely) oriented boundary curve is $\Gamma_R(t_0) := \curbr{ t \in \C : |t-t_0| = R }$, and $\mathcal{P}_R := \mathcal{P} \cap D_R(t_0)$ refers to the (finite open) subsegment of $\mathcal{P}$ that is contained in $D_R(t_0)$.
\\
\hspace*{1em}It is well-known that analyticity of a function $\theta : D_R(t_0) \rightarrow \C$, with fixed $R > 0$, is equivalent to uniform convergence in $D_{r_1}(t_0)$ of the Taylor series $\theta(t) = \sum_{k=0}^\infty (k!)^{-1} \theta^{(k)}(t_0) (t-t_0)^k$, for each $0 < r_1 < R$. The point $t_0$ is a zero of $\theta(t)$ of order (or multiplicity) $K\in \N$, if $\theta^{(k)}(t_0) = 0$, for all $0 \leq k \leq K-1$, and $\theta^{(K)}(t_0)\neq0$. If $K \in \N\setminus\curbr{1}$, $\theta^{(k)}(t_0) = 0$, for each $1 \leq k \leq K-1$, and $\theta^{(K)}(t_0)\neq0$, then $t_0$ is called a saddle point of order $K-1$. In these circumstances, the function amplifies by $K$-times the angle between any two curves intersecting at $t = t_0$. Equivalently, a small neighborhood of $t = t_0$ is mapped to a Riemann surface with $K$ sheets. On each sheet, $\theta(t)$ attains each value once, and a path of steepest descent emanates from $t_0$, i.e., a line along which $\Im \theta$ is constant. The principle behind the method of steepest descent is that the asymptotic evaluation of a generalized Laplace transform with phase function $\theta(t)$ essentially simplifies, if a deformation of its integration path to a path of steepest descent is applicable (see \cite{Olver1974} or \cite{paris_2011}).
\\
\hspace*{1em}Of special importance is the case that $\theta'(t_0) \neq 0$, in which $\theta : D_R(t_0) \rightarrow \C$ is angle-preserving, called (locally) conformal. Then, according to the inverse function theorem (cf. Theorem 5.7.13 in \cite{Asmar2018}), it exists $0 < r_2 < R$, such that $\theta$ is one-to-one on the disc $D_{r_2}(t_0)$. The radius of the greatest admissible disc naturally depends on the properties of $\theta$. Landau's estimate (see Exercise 35 in \cite{Asmar2018}) asserts that it suffices to choose $r_2 := (4M)^{-1} R^2 |\theta'(t_0)|$, with $M := \max_{z \in \Gamma_R(t_0)} |\theta(t)|$. Furthermore, $\theta$ then even has an analytic local inverse. To be more precise, in a small neighborhood of $w_0 := \theta(t_0)$, the function $\theta^{-1} : \theta(D_{r_2}(t_0)) \rightarrow D_{r_2}(t_0)$ can be expanded in powers of $\theta(t)-w_0$. This representation is known as the Lagrange inversion formula (compare Theorem 5.7.17 and Exercise 32 in \cite{Asmar2018} or $\S$6.23 in \cite{copson1970intro}). With that in mind, it becomes obvious that any function $f(t)$ that is analytic in a (possibly punctured) neighborhood of $t=t_0$ admits an expansion in powers of an arbitrary conformal function $\theta(t)$, where $\theta(t) \equiv t$ corresponds to the well-known Taylor or Laurent expansion. Informally speaking, it exist $(\operatorname{c}(k))_{k \in \Z} \subset \C$ and $r_3 > 0$, for which
\begin{align*}
f(t) = \suml_{k = -\infty}^\infty \curbr{\theta(t)-\theta(t_0)}^k \operatorname{c}(k) \hspace{1cm} (t \in D_{r_3}(t_0)).
\end{align*}
Bürmann (cf. \cite{Buermann1799} or $\S$7.3 in \cite{whittakerwatson1952}) first established a finite version of this expansion for analytic functions. A later extension due to Teixeira (see \cite{Teixeira1900} or $\S$7.31 in \cite{whittakerwatson1952}) also covers functions with an isolated singularity at $t=t_0$ and even includes a sufficient condition for convergence.

\section{Analytic continuation of a generalized Mellin transform with analytic ingredients} \label{AnContTeixExp}

We begin with a study of the generating function $\mathcal{M}_{\varphi, \mathcal{P} }\curbr{a}(z)$. Since $\mathcal{M}_{\varphi, \mathcal{P} \setminus \mathcal{P}_r }\curbr{a}(z)$ is entire, for any path $\mathcal{P}$ as in Assumption \ref{AnContTeixExpA} but with finite $T \in \R$, continuous functions $\varphi : \mathcal{P} \cup \curbr{T} \rightarrow (0, \infty)$, $a : \mathcal{P} \cup \curbr{T} \rightarrow \C$ and $r>0$, it is clear that the region of analyticity of $\mathcal{M}_{\varphi, \mathcal{P} }\curbr{a}(z)$ is merely determined by the behaviour of the ingredient functions near $t=t_0$. For that reason, it absolutely suffices to confine to a study of $\mathcal{M}_{\varphi, \mathcal{P}_r }\curbr{a}(z)$, where an appropriate radius $r>0$ is essentially characterized by the following assumption.

\begin{ass} \label{AssIngGenFct2}
It exists $R>0$, with $\varphi(t) > 0$, for $t \in \mathcal{P}_R$. In addition, it exist $\alpha_0 \in \C$ with $\Re \alpha_0 > -1$, $\beta_0 > 0$ and analytic functions $q, p : D_R(t_0) \rightarrow \C$, where $p(t_0) \neq 0$, such that, for $t \in \mathcal{P}_R \cup \curbr{t_0}$, we have the factorizations
\begin{align*}
a(t) &= (t-t_0)^{\alpha_0} q(t), \\
\varphi(t) &= (t-t_0)^{\beta_0} p(t).
\end{align*}
In this, fractional powers $\gamma \in \curbr{\alpha_0, \beta_0}$ are obtained from (\ref{ZeroAlgBound6}) via the relation $(t-t_0)^\gamma \sim |t-t_0|^\gamma e^{i\gamma \omega_\mathcal{P}}$, as $t \rightarrow t_0$ with $t \in \mathcal{P}_R \cup \curbr{t_0}$, and elsewhere on $\mathcal{P}_R$ by continuity.
\end{ass}

Obviously, by construction, $a(t)$ and $\varphi(t)$ are also analytic in a region $D \subset \C$, however, not necessarily at $t=t_0$. In fact, for fractional $\alpha_0$ or $\beta_0$, it is an algebraic branch point of the respective function and must hence lie on the boundary of $D$. Conversely, for $\alpha_0 \in \N_0$ and $\beta_0 \in \N$, we infer that $D_R(t_0) \subseteq D$. Now, denoting
\begin{align} \label{ZeroAlgBound2c}
\chi_n := \frac{\alpha_0+1}{\beta_0} + \frac{n}{\beta_0} \hspace{1cm} (n\in\N_0),
\end{align}
as a consequence of Assumptions \ref{AnContTeixExpA}, \ref{AssIngGenFct2} and Lemma \ref{LemAnGenFctAlgKer}, for each $0 < r < R$, the integral definition of the generating function $\mathcal{M}_{\varphi, \mathcal{P}_r }\curbr{a}(z)$ (see (\ref{AsymptExpByAnalyticCont10})) converges absolutely and is holomorphic in the half plane $\Re z < \Re\chi_0$. The traditional method for analytically continuing an ordinary Mellin transform relies on a (formal) power series expansion for the amplitude (cf. $\S$4.3 in \cite{bleisteinhandelsman1986}). In an analogous fashion, with regard to $\mathcal{M}_{\varphi, \mathcal{P}_r }\curbr{a}(z)$, we aim to expand
\begin{align} \label{ZeroAlgBound0}
a_\varphi(t) := \frac{a(t)}{\varphi'(t)}
\end{align}
in powers of $\varphi(t)$, which will eventually facilitate a partial evaluation of the integral by means of the fundamental theorem of calculus. Yet, since $a_\varphi(t)$ need not be analytic at $t=t_0$, we can not expect the existence of an expansion in a whole neighborhood of this point, so that we rather confine to an expansion that is valid merely for $t \in \mathcal{P}_r$, with small $0 < r < R$. It will be obtained after proper factorization of $a_\varphi(t)$ into a fractional and an analytic part. First, however, notice that $\varphi(t)$ actually is conformal at $t=t_0$, merely if $\beta_0=1$. For that reason, we construct an auxiliary conformal function whose $\beta_0$-th power coincides with $\varphi(t)$, for $t \in \mathcal{P}_r$. We then expand $a_\varphi(t)$ in powers of this function, which corresponds to an expansion in powers of $\curbr{\varphi(t)}^\frac{1}{\beta_0}$.
\\
\hspace*{1em}When working with powers of analytic functions, it is always important to choose the correct branches of the arguments. Here, the situation is very convenient, because $\varphi(t)>0$, for $t \in \mathcal{P}_R$, so that each branch of the argument of $\varphi(t)$ and of $p(t)$ is constant along $\mathcal{P}_R$. Since the kernel of $\mathcal{M}_{\varphi, \mathcal{P} }\curbr{a}(z)$ was defined in terms of the principal branch of the argument of $\varphi(t)$, it particularly follows that $\arg\curbr{\varphi(t)}=0$, for $t \in \mathcal{P}_R$, and $\arg\curbr{\varphi(t_0)}=0$, by continuity. From this and Assumption \ref{AssIngGenFct2}, we infer that
\begin{align} \label{ZeroAlgBound7}
\arg\curbr{p(t)} = -\beta_0\omega_\mathcal{P} \hspace{1cm} (t \in \mathcal{P}_R \cup \curbr{t_0}).
\end{align}
Based on this branch, we construct any fractional powers of $p(t)$, i.e., for $\kappa \in \C$, we agree that
\begin{align*}
\curbr{p(t)}^\kappa = |p(t)|^\kappa e^{-i\kappa\beta_0\omega_\mathcal{P}} \hspace{1cm} (t \in \mathcal{P}_R \cup \curbr{t_0}).
\end{align*}
By assumption, the function $\curbr{p(t)}^\kappa$ is even analytic in a neighborhood of $\mathcal{P}_R \cup \curbr{t_0}$, since $p(t_0) \neq 0$ and because $\varphi(t) > 0$  implies that $p(t) \neq 0$, for $t \in \mathcal{P}_R$. Let $D_{r_p}(t_0)$ be the greatest disc of radius $0 < r_p < R$ around $t_0$, in which this power function remains analytic. Such a disc exists, due to the fact that zeros and singularities of analytic functions lie isolated. Indeed, $r_p$ is determined by the closest zero or singularity of $p(t)$ (cf. Lemma 3.6.11 in \cite{wegert2012visual}). Finally, also
\begin{align} \label{ZeroAlgBound9}
\varphi_\mathcal{P}(t) := (t-t_0) \rrb{ p(t) }^\frac{1}{\beta_0}
\end{align}
is analytic in $D_{r_p}(t_0)$. Since $\varphi_\mathcal{P}(t_0)=0$ and $\varphi_\mathcal{P}'(t_0) \neq 0$, the function is even conformal at $t=t_0$. As a consequence, it exists $0 < r_\varphi < R$, such that $\varphi_\mathcal{P}(t_1) = \varphi_\mathcal{P}(t_2)$, for $t_1, t_2 \in D_{r_\varphi}(t_0)$, if and only if $t_1=t_2$. Moreover, we observe that $|\varphi_\mathcal{P}(t)|^{\beta_0} = \varphi(t)$, for any $t \in \mathcal{P}_R$, as well as that
\begin{align} \label{ZeroAlgBound10}
\beta_0\arg\rrb{\varphi_\mathcal{P}(t_0) } = \beta_0\omega_\mathcal{P} + \arg\curbr{p(t_0)} = 0.
\end{align}
Therefore, by continuity, we conclude that $\curbr{\varphi_\mathcal{P}(t)}^{\beta_0} = \varphi(t)$ and conversely that
\begin{align} \label{ZeroAlgBound12A}
\varphi_\mathcal{P}(t) = \rrb{\varphi(t)}^\frac{1}{\beta_0} > 0 \hspace{1cm} (t \in \mathcal{P}_R).
\end{align}
Clearly, $\varphi_\mathcal{P}(t)$ constitutes a $\beta_0$-th root of $\varphi(t)$ at $t=t_0$, if $\beta_0 \in \N\setminus\curbr{1}$. We can now proceed with the indicated decomposition of $a_\varphi(t)$. Due to Assumption \ref{AssIngGenFct2}, in terms of
\begin{align} \label{2026052201}
b(t) := \beta_0 p(t) + (t-t_0)p'(t),
\end{align}
we have $\varphi'(t) = (t-t_0)^{\beta_0-1} b(t)$. Furthermore, because $b(t_0) \neq 0$, we can find $0 < r_b < R$, with $b(t)\neq 0$, for all $t \in D_{r_b}(t_0)$. In these circumstances, $\curbr{ b(t) }^{-1} q(t)$ is holomorphic in the disc $D_{r_b}(t_0)$ and by definition (\ref{ZeroAlgBound0}), we can write
\begin{align*}
a_\varphi(t) = (t-t_0)^{\alpha_0+1-\beta_0} \frac{q(t)}{b(t)} \hspace{1cm} (t \in \mathcal{P}_{r_b}).
\end{align*}
We also introduce the function
\begin{align} \label{2026052202}
A_\mathcal{P}(t) := \curbr{ p(t) }^{1-\chi_0} \frac{q(t)}{b(t)}.
\end{align}
It is analytic in the disc $D_{r_A}(t_0)$, for $r_A := \min\curbr{r_p, r_b}$. By construction of $\varphi_\mathcal{P}(t)$ in (\ref{ZeroAlgBound9}), it moreover fulfills
\begin{align*}
A_\mathcal{P}(t) = \rrb{ \frac{ t-t_0 }{ \varphi_\mathcal{P}(t) } }^{\alpha_0+1-\beta_0} \frac{q(t)}{b(t)} \hspace{1cm} (t \in \mathcal{P}_{r_A}).
\end{align*}
Accordingly, in terms of this function, we eventually arrive at
\begin{align} \label{2026052203}
a_\varphi(t) = \rrb{ \varphi_\mathcal{P}(t) }^{\beta_0(\chi_0-1)} A_\mathcal{P}(t) \hspace{1cm} (t \in \mathcal{P}_{r_A}).
\end{align}
On the one hand, the first factor is analytic in a region of which $t_0$ is a boundary point. On the other hand, due to its analyticity, the function $A_\mathcal{P}(t)$ admits an expansion in powers of $\varphi_\mathcal{P}(t)$, from what we will altogether obtain our expansion for $a_\varphi(t)$. More precisely, for $N \in \N$ and $0 < r_0 < \min\curbr{r_A, r_\varphi}$, defining
\begin{align} \label{ZeroAlgBound14XY}
\operatorname{c}_\mathcal{P}(n) := \frac{1}{\beta_0} \frac{1}{2\pi i} \ointl_{\Gamma_{r_0}(t_0)} \frac{ q(w) }{ \curbr{p(w)}^\frac{\alpha_0}{\beta_0} } \frac{ dw }{ \rrb{\varphi_\mathcal{P}(w)}^{n+1} } \hspace{1cm} (0 \leq n \leq N-1),
\end{align}
as well as
\begin{align} \label{ZeroAlgBound14d}
\operatorname{C}_\mathcal{P}(t, N) := \rrb{ \varphi_\mathcal{P}(t) }^N \frac{1}{\beta_0} \frac{1}{2\pi i} \ointl_{\Gamma_{r_0}(t_0)} \frac{q(w) }{\curbr{p(w)}^\frac{\alpha_0}{\beta_0} \rrb{ \varphi_\mathcal{P}(w)}^N } \frac{dw}{ \varphi_\mathcal{P}(w) - \varphi_\mathcal{P}(t) },
\end{align}
the following expansion holds.

\begin{lem} \label{LemZeroAlgBound15}
Under Assumptions \ref{AnContTeixExpA} and \ref{AssIngGenFct2}, for all $(t, N) \in \mathcal{P}_{r_0} \times \N_0$, we have
\begin{align*}
a_\varphi(t) = \suml_{n= 0}^{N-1} \rrb{\varphi_\mathcal{P}(t)}^{\beta_0(\chi_n-1)} \operatorname{c}_\mathcal{P}(n) + \rrb{\varphi_\mathcal{P}(t)}^{\beta_0(\chi_0-1)} \operatorname{C}_\mathcal{P}(t, N).
\end{align*}
If $|\varphi_\mathcal{P}(t)| < \min_{w \in \Gamma_{r_0}(t_0)} |\varphi_\mathcal{P}(w)|$, then $\operatorname{C}_\mathcal{P}(t, N) = o(1)$, as $N \rightarrow \infty$.
\end{lem}

We emphasize that the exact region of convergence of the sequence of partial sums substantially depends on the mapping properties of $\varphi_\mathcal{P}(t)$. However, in any case, the required condition is fulfilled at least on a small subsegment of $\mathcal{P}_{r_0}$, close to $t=t_0$.

\begin{proof}[Proof of Lemma \ref{LemZeroAlgBound15}]
Notice that the closed disc $\overline {D}_{r_0}(t_0)$ is included in the region of analyticity of $A_\mathcal{P}(t)$. Therefore, by Cauchy's theorem, for any $t \in D_{r_0}(t_0)$, we have
\begin{align*}
A_\mathcal{P}(t) = \frac{1}{2\pi i} \ointl_{\Gamma_{r_0}(t_0)} \frac{A_\mathcal{P}(w)}{w-t} dw.
\end{align*}
In addition, recall that $\varphi_\mathcal{P}(t)$ is analytic and one-to-one in $D_{r_\varphi}(t_0)$, and observe that $\overline{D}_{r_0}(t_0) \subset D_{r_\varphi}(t_0)$. Hence, for each $t \in D_{r_0}(t_0)$, even $(\varphi_\mathcal{P}(w)-\varphi_\mathcal{P}(t))^{-1}\varphi_\mathcal{P}'(w)$ is $w$-analytic there, except at $w=t$, where it expands to $(w-t)^{-1} + \LandauO(1)$. Consequently, for any $t \in D_{r_0}(t_0)$, we deduce that
\begin{align*}
A_\mathcal{P}(t) &= \frac{1}{2\pi i} \ointl_{\Gamma_{r_0}(t_0)} \frac{A_\mathcal{P}(w)\varphi_\mathcal{P}'(w)}{\varphi_\mathcal{P}(w)-\varphi_\mathcal{P}(t)} dw + \frac{1}{2\pi i} \ointl_{\Gamma_{r_0}(t_0)} A_\mathcal{P}(w) \rrb{ \frac{1}{w-t} - \frac{\varphi_\mathcal{P}'(w)}{\varphi_\mathcal{P}(w)-\varphi_\mathcal{P}(t)} } dw \\
&= \frac{1}{2\pi i} \ointl_{\Gamma_{r_0}(t_0)} \frac{A_\mathcal{P}(w)\varphi_\mathcal{P}'(w)}{\varphi_\mathcal{P}(w)-\varphi_\mathcal{P}(t)} dw.
\end{align*}
In particular, $\varphi_\mathcal{P}(w) \neq \varphi_\mathcal{P}(t)$ and $\varphi_\mathcal{P}(w) \neq 0$, for $(w, t) \in \Gamma_{r_0}(t_0) \times D_{r_0}(t_0)$. With the aid of the formula for geometric sums, for $(w, t, N) \in \Gamma_{r_0}(t_0) \times D_{r_0}(t_0) \times \N$, it is thus possible to write
\begin{align*}
\frac{1}{ \varphi_\mathcal{P}(w) - \varphi_\mathcal{P}(t) } = \suml_{n=0}^{N-1} \frac{ \rrb{ \varphi_\mathcal{P}(t) }^n }{ \rrb{\varphi_\mathcal{P}(w)}^{n+1} } + \frac{ \rrb{ \varphi_\mathcal{P}(t) }^N }{ \rrb{\varphi_\mathcal{P}(w)}^N } \frac{ 1 }{ \varphi_\mathcal{P}(w)-\varphi_\mathcal{P}(t) }.
\end{align*}
Furthermore, according to (\ref{ZeroAlgBound9}) and (\ref{2026052201}), we compute
\begin{align*}
\varphi_\mathcal{P}'(w) = \frac{1}{\beta_0} \curbr{p(w)}^{\frac{1}{\beta_0}-1} b(w) \hspace{1cm} (w \in D_{r_A}(t_0)).
\end{align*}
Thus, by definition of $A_\mathcal{P}(w)$ in (\ref{2026052202}), we get
\begin{align*}
A_\mathcal{P}(w) \varphi_\mathcal{P}'(w) = \frac{1}{\beta_0} \frac{q(w)}{\rrb{p(w)}^\frac{\alpha_0}{\beta_0}} \hspace{1cm} (w \in D_{r_A}(t_0)).
\end{align*}
Upon combining these findings with the above expansion and the representation (\ref{2026052203}) for $a_\varphi(t)$, as well as definition (\ref{ZeroAlgBound2c}), we eventually arrive at the asserted expansion, with the coefficients and the remainder integral as given in (\ref{ZeroAlgBound14XY}) and (\ref{ZeroAlgBound14d}), respectively. With regard to the latter, denoted by $\operatorname{C}_\mathcal{P}(t, N)$, since $\overline{D}_{r_0}(t_0)$ is included in a domain where $\varphi_\mathcal{P}(w)$ is one-to-one, it holds that $\min_{w \in \Gamma_{r_0}(t_0)} |\varphi_\mathcal{P}(w)-\varphi_\mathcal{P}(t)| > 0$, for each $t \in D_{r_0}(t_0)$. Thus, if $|\varphi_\mathcal{P}(t)| < \min_{w \in \Gamma_{r_0}(t_0)} |\varphi_\mathcal{P}(w)|$, it becomes obvious from (\ref{ZeroAlgBound14d}) that $\operatorname{C}_\mathcal{P}(t, N) = o(1)$, as $N \rightarrow \infty$, which completes the proof.
\end{proof}

For actually computing the coefficients (\ref{ZeroAlgBound14XY}), it may be helpful to recast them in terms of derivatives of the involved functions. To be more precise, from Cauchy's theorem, we deduce that
\begin{align*}
\operatorname{c}_\mathcal{P}(n) = \frac{1}{\beta_0} \frac{1}{n!} \frac{d^n}{dw^n} \frac{ q(w) (w-t_0)^{n+1} }{ \curbr{p(w)}^\frac{\alpha_0}{\beta_0} \rrb{\varphi_\mathcal{P}(w)}^{n+1} } \Bigg|_{w=t_0} \hspace{1cm} (0\leq n \leq N-1).
\end{align*}
Thereof, according to the definition of $\chi_n$ and $\varphi_\mathcal{P}(t)$ in (\ref{ZeroAlgBound2c}) and (\ref{ZeroAlgBound9}), we get
\begin{align} \label{ZeroAlgBound21b}
\operatorname{c}_\mathcal{P}(n) = \frac{1}{\beta_0} \frac{1}{n!} \frac{d^n}{dw^n} \frac{ q(w) }{ \curbr{ p(w) }^{\chi_n} } \Bigg|_{w=t_0}.
\end{align}
Specifically $\operatorname{c}_\mathcal{P}(0) = \beta_0^{-1} \curbr{p(t_0)}^{-\chi_0} q(t_0)$. By means of the expansion from Lemma \ref{LemZeroAlgBound15}, we can finally next determine the analytic continuation of the generating function $\mathcal{M}_{\varphi, \mathcal{P}_r }\curbr{a}(z)$, whenever $0 < r < r_0$. It is represented by a decomposition into a singular and an analytic part, where the singularities are expanded in increasing order with respect to the real part. We denote the associated remainder term by
\begin{align} \label{ApplicOfExp3}
\mathcal{M}_{\varphi, \mathcal{P}_r}^N\curbr{a}(z) := \intl_{\mathcal{P}_r} \rrb{\varphi(t)}^{-z} \rrb{\varphi_\mathcal{P}(t)}^{\beta_0(\chi_0-1)} \operatorname{C}_\mathcal{P}(t, N) \varphi'(t) dt,
\end{align}
with $\operatorname{C}_\mathcal{P}(t, N)$ referring to the integral from (\ref{ZeroAlgBound14d}).

\begin{theo} \label{TheoAnContAnIng}
Under Assumptions \ref{AnContTeixExpA} and \ref{AssIngGenFct2}, for each $0 < r < r_0$ and $N\in\N$, the integral definition of the generating function $\mathcal{M}_{\varphi, \mathcal{P}_r }\curbr{a}(z)$ can be continued to a meromorphic function in the half plane $\Re z < \Re\chi_N$, through the expansion
\begin{align} \label{ApplicOfExp4}
\mathcal{M}_{\varphi, \mathcal{P}_r }\curbr{a}(z) = - e^{i\omega_\mathcal{P}} \suml_{n=0}^{N-1} \operatorname{c}_\mathcal{P}(n) \frac{\rrb{\varphi(r)}^{\chi_n-z} }{z-\chi_n} + \mathcal{M}_{\varphi, \mathcal{P}_r}^N\curbr{a}(z),
\end{align}
whose coefficients can be found in (\ref{ZeroAlgBound14XY}) or (\ref{ZeroAlgBound21b}), while the remainder term is available in (\ref{ApplicOfExp3}). In the indicated region, the only possible singularities are the sequence of simple poles $( \chi_n )_{0 \leq n \leq N-1}$, with residues
\begin{align} \label{ApplicOfExp5}
\Res{ z= \chi_n }{ \mathcal{M}_{\varphi, \mathcal{P}_r }\curbr{a}(z) } = - e^{i\omega_\mathcal{P}} \operatorname{c}_\mathcal{P}(n) \hspace{1cm} (0 \leq n \leq N-1).
\end{align}
Lastly, $\mathcal{M}_{\varphi, \mathcal{P}_r }\curbr{a}(z) = \LandauO\curbr{ |z|^{-1} }$, as $\Im z \rightarrow\pm\infty$, in the half plane $\Re z < \Re\chi_N$.
\end{theo}

Notice that the residue is independent from the radius $r$. Also, poles may turn out as removable singularities, whenever $\operatorname{c}_\mathcal{P}(n)=0$, for some $0 \leq n \leq N-1$.

\begin{proof}[Proof of Theorem \ref{TheoAnContAnIng}]
Throughout the proof, $0 < r < r_0$ and $N\in \N_0$ are fixed but arbitrary. Now, recall first that the integral definition of $\mathcal{M}_{\varphi, \mathcal{P}_r }\curbr{a}(z)$ converges absolutely and is analytic in $\Re z < \Re\chi_0$, compare (\ref{AsymptExpByAnalyticCont10}). With the aid of (\ref{ZeroAlgBound0}), it becomes
\begin{align*}
\mathcal{M}_{\varphi, \mathcal{P}_r }\curbr{a}(z) = \intl_{\mathcal{P}_r} \rrb{\varphi(t)}^{-z} a_\varphi(t) \varphi'(t) dt.
\end{align*}
By construction, $\overline{D}_r(t_0) \subset D_{r_0}(t_0)$, which implies that $\mathcal{P}_r \subset \mathcal{P}_{r_0}$. We may hence expand $a_\varphi(t)$ in the fashion of Lemma \ref{LemZeroAlgBound15}. Thereof, in terms of the integral (\ref{ApplicOfExp3}), for $\Re z < \Re\chi_0$, we get
\begin{align*}
\mathcal{M}_{\varphi, \mathcal{P}_r }\curbr{a}(z) = \suml_{n=0}^{N-1} \operatorname{c}_\mathcal{P}(n) \intl_{\mathcal{P}_r} \rrb{\varphi(t)}^{-z} \rrb{\varphi_\mathcal{P}(t)}^{\beta_0(\chi_n-1)} \varphi'(t) dt + \mathcal{M}_{\varphi, \mathcal{P}_r}^N\curbr{a}(z).
\end{align*}
By Lemma \ref{LemAnGenFctAlgKer} and since $\Re\chi_n \geq \Re\chi_0$, for each $0 \leq n \leq N-1$, the integral in the $n$-th summand converges absolutely. In view of the fundamental theorem of calculus and the identity (\ref{ZeroAlgBound12A}), it simplifies to
\begin{align*}
\intl_{\mathcal{P}_r} \rrb{\varphi(t)}^{-z} \rrb{\varphi_\mathcal{P}(t)}^{\beta_0(\chi_n-1)} \varphi'(t) dt = \intl_{\mathcal{P}_r} \rrb{\varphi(t)}^{\chi_n -z-1} \varphi'(t) dt = e^{i\omega_\mathcal{P}} \frac{\rrb{\varphi(r)}^{\chi_n-z} }{\chi_n-z}.
\end{align*}
It shows that the integral on the left hand side can be extended to the whole complex plane as a meromorphic function, whose only singularity is a simple pole at $z=\chi_n$. Regarding the remainder integral, for convenience, we define
\begin{align*}
F_N(t) := \frac{1}{2\pi i} \ointl_{ \Gamma_{r_0}(t_0) } \frac{ q(w) }{ \curbr{p(w)}^\frac{\alpha_0}{\beta_0} \rrb{\varphi_\mathcal{P}(w)}^N } \frac{ dw }{ \varphi_\mathcal{P}(w) - \varphi_\mathcal{P}(t) }.
\end{align*}
Then, according to (\ref{ZeroAlgBound2c}), (\ref{ZeroAlgBound12A}), (\ref{ZeroAlgBound14d}) and (\ref{ApplicOfExp3}), we can write
\begin{align*}
\mathcal{M}_{ \varphi, \mathcal{P}_r }^N\curbr{a}(z) = \frac{1}{\beta_0} \intl_{ \mathcal{P}_r } \rrb{\varphi(t)}^{\chi_N-z-1} \varphi'(t) F_N(t) dt.
\end{align*}
Since $\varphi_\mathcal{P}(t)$ is analytic and one-to-one in $D_{r_\varphi}(t_0)$ and $\overline{D}_{r_0}(t_0) \subset D_{r_\varphi}(t_0)$, the integral $F_N(t)$ establishes an analytic function in $D_{r_0}(t_0)$, whose derivatives can be obtained from differentiation under the integral sign (see Theorem 5.6.1 in \cite{wegert2012visual}). At the same time, recall that $\mathcal{P}_r$ runs in the interior of $D_{r_0}(t_0)$. We hence conclude that $F_N(t)$ is infinitely many times continuously differentiable on $\mathcal{P}_r$, with $\sup_{t \in \mathcal{P}_r} |F_N^{(k)}(t)| = \LandauO(1)$, for each $k \in \N_0$. Moreover, $\sup_{t \in \mathcal{P}_r} |t-t_0|^{\beta_0}\rrb{\varphi(t)}^{-1} = \LandauO(1)$ and $\rrb{\varphi(t)}^{\chi_N-1} \varphi'(t) = \LandauO(|t-t_0|^{\beta_0\Re\chi_N-1})$, as $t \rightarrow t_0$, with $t \in \mathcal{P}_r$. By Lemma \ref{LemAnGenFctAlgKer}, absolute convergence and analyticity of $\mathcal{M}_{\varphi, \mathcal{P}_r}^N\curbr{a}(z)$ is thus guaranteed for $\Re z < \Re\chi_N$, so that altogether the whole expansion for $\mathcal{M}_{\varphi, \mathcal{P}_r }\curbr{a}(z)$ is meromorphic there. Notice that this expansion indeed represents an analytic continuation of the original integral (\ref{AsymptExpByAnalyticCont10}), since $\Re \chi_N > \Re\chi_0$. The residue at each of the simple poles $(\chi_n)_{0 \leq n \leq N-1}$ is obvious. It remains to examine the asymptotic behaviour of the expansion. Clearly, for $0 \leq n \leq N-1$, the $n$-th summand is $\LandauO\curbr{ \abs{z}^{-1} }$, as $\Im z \rightarrow \pm\infty$. To determine the behaviour of $\mathcal{M}_{\varphi, \mathcal{P}_r}^N\curbr{a}(z)$, appealing to the continuous differentiability of $F_N(t)$, we integrate by parts, for $\Re z < \Re \chi_N$, yielding
\begin{align*}
\mathcal{M}_{\varphi, \mathcal{P}_r}^N\curbr{a}(z) = e^{i\omega_\mathcal{P}} \frac{ \rrb{\varphi(r)}^{\chi_N-z} }{ \beta_0 (\chi_N-z) } F_N(r) - \frac{ 1 }{ \beta_0(\chi_N-z) } \intl_{\mathcal{P}_r} \rrb{\varphi(t)}^{\chi_N-z} F_N'(t) dt.
\end{align*}
By Lemma \ref{LemAnGenFctAlgKer}, the integral on the right hand side converges absolutely and is holomorphic in the half plane $\Re z < \Re\chi_{N+1}$. Hence, $\mathcal{M}_{\varphi, \mathcal{P}_r}^N\curbr{a}(z) = \LandauO\curbr{ |z|^{-1} }$, as $\Im z\rightarrow \pm\infty$, and the proof is finished.
\end{proof}

\section{Expansions for generalized Laplace transforms with a large argument} \label{SecDerivAsympStatI}

We eventually deploy our findings from $\S$\ref{AnContTeixExp} for the derivation of asymptotic expansions for the integral $\mathcal{L}_{ \phi, \mathcal{P} }\curbr{a}(\lambda)$. The associated ingredient functions are supposed to have the following properties.

\begin{ass} \label{AssIngFctLT}
$\phi(t) > \phi(t_0) \geq 0$, for any $t \in \mathcal{P} \cup \curbr{T}$, and it exists $\lambda_0 \geq 0$, with $\int_\mathcal{P} \exp\curbr{-\lambda_0\phi(t)} |a(t)| dt < \infty$.
\end{ass}

Notice that $\phi(T)$ need not be finite. The transition from $\mathcal{L}_{ \phi, \mathcal{P} }\curbr{a}(\lambda)$ to $\mathcal{M}_{ \varphi, \mathcal{P} }\curbr{a}(z)$ is managed by the next condition, according to which the Laplacian kernel can be represented as a Mellin-Barnes integral.

\begin{ass} \label{LemExpwrtAsympScaleA}
For any fixed $\lambda > 0$, it exist $\zeta_\lambda > 0$ and $\Phi : \C \times (0, \infty) \rightarrow \C$, with
\begin{align*}
e^{-\lambda \phi(t)} = \frac{1}{2\pi i} \intl_{x_0-i\infty}^{x_0+i\infty} \Phi(z, \lambda) \curbr{\varphi(t)}^{-z} dz \hspace{1cm} (t \in \mathcal{P}),
\end{align*}
for each $0 < x_0 < \zeta_\lambda$. In this, $\zeta_\lambda \rightarrow \infty$, as $\lambda \rightarrow \infty$. Besides, $\Phi(z, \lambda)$, for fixed $\lambda>0$, is $z$-holomorphic in the strip $0 < \Re z < \zeta_\lambda$ and $\Phi(z,\lambda) = \LandauO\curbr{ |z|^{-2} }$, as $\Im z \rightarrow \pm\infty$ there. For fixed $0 < \Re z_1 < \Re z_2 < \zeta_\lambda$, we also have
\begin{align*}
\Phi(z_2, \lambda) = o\rrb{ \Phi(z_1, \lambda) } \hspace{1cm} (\lambda \rightarrow \infty),
\end{align*}
as well as, for each $x > 0$,
\begin{align*}
\intl_{-\infty}^\infty \frac{ |\Phi(x+iy, \lambda)| }{ |x+iy| } dy = \LandauO\rrb{\Phi(x, \lambda)} \hspace{1cm} (\lambda \rightarrow \infty).
\end{align*}
Lastly, for each $c > 0$, it holds that $e^{-c \lambda} = o\curbr{ \Phi(z, \lambda) }$, as $\lambda \rightarrow \infty$, for all $\Re z > 0$.
\end{ass}

The function $\Phi(\cdot, \lambda)$ generates the asymptotic sequence, in terms of which we finally wish to expand $\mathcal{L}_{ \phi, \mathcal{P} }\curbr{a}(\lambda)$. According to the Mellin inversion formula, if $\varphi$ is monotonic along $\mathcal{P}$, as in all our applications below, $\Phi(\cdot, \lambda)$ constitutes the Mellin transform of $s \mapsto \exp\curbr{-\lambda(\phi\circ\varphi^{-1})(s)}$. We are now ready for the main theorem. Based on the residue theorem, it provides the foundation for numerous subsequent results.

\begin{theo} \label{TheoExpwrtAsympScale}
Under Assumptions \ref{AnContTeixExpA} and \ref{AssIngFctLT}, if it exists a function $\varphi : \mathcal{P} \rightarrow (0, \infty)$, such that additionally Assumptions \ref{AssIngGenFct2} and \ref{LemExpwrtAsympScaleA} hold, then
\begin{align*}
\mathcal{L}_{ \phi, \mathcal{P} }\curbr{a}(\lambda) \sim e^{i \omega_\mathcal{P} } \suml_{n=0}^\infty \Phi(\chi_n, \lambda) \operatorname{c}_\mathcal{P}(n) \hspace{1cm} (\lambda \rightarrow \infty).
\end{align*}
The parameters $( \chi_n )_{n \in \N_0}$ were given in (\ref{ZeroAlgBound2c}), and the coefficients $( \operatorname{c}_\mathcal{P}(n) )_{n \in \N_0}$ are available in (\ref{ZeroAlgBound14XY}) or (\ref{ZeroAlgBound21b}).
\end{theo}

In the theory of Mellin-Barnes integrals, the step in the proof below, in which the initial integral gets replaced by a sum of residues and a remainder integral, is often concisely referred to as a displacement of the integration path.

\begin{proof}[Proof of Theorem \ref{TheoExpwrtAsympScale}]
As a consequence of Assumption \ref{AssIngFctLT}, the integral definition of the Laplace transform $\mathcal{L}_{ \phi, \mathcal{P} }\curbr{a}(\lambda)$, compare (\ref{AsymptExpByAnalyticCont1}), converges absolutely, for all $\lambda \geq \lambda_0$. Moreover, due to Assumption \ref{AssIngGenFct2}, it exists a parameter $r$ as in Theorem \ref{TheoAnContAnIng}. Then, the integral de\-fi\-nition of the generating function $\mathcal{M}_{\varphi, \mathcal{P}_r }\curbr{a}(z)$, see (\ref{AsymptExpByAnalyticCont10}), converges absolutely and is holomorphic in $\Re z < \Re\chi_0$. Upon writing the kernel of $\mathcal{L}_{\phi, \mathcal{P}_r }\curbr{a}(\lambda)$ in terms of the integral representation from Assumption \ref{LemExpwrtAsympScaleA}, for $0 < x_0 < \min\curbr{\Re\chi_0, \zeta_\lambda}$, especially by absolute convergence, we may thus interchange the order of integration, leading to the Mellin-Barnes representation
\begin{align*}
\mathcal{L}_{\phi, \mathcal{P}_r }\curbr{a}(\lambda) &= \frac{1}{2\pi i} \intl_{x_0-i\infty}^{x_0+i\infty} \Phi(z, \lambda) \intl_{\mathcal{P}_r} \rrb{\varphi(t)}^{-z} a(t) dt dz \\
&= \frac{1}{2\pi i} \intl_{x_0-i\infty}^{x_0+i\infty} \Phi(z, \lambda) \mathcal{M}_{\varphi, \mathcal{P}_r }\curbr{a}(z) dz.
\end{align*}
With regard to the analytic continuation of $\mathcal{M}_{\varphi, \mathcal{P}_r }\curbr{a}(z)$, we may directly refer to Theorem \ref{TheoAnContAnIng}. Accordingly, the function can be extended meromorphically to the half plane $\Re z < \Re\chi_N$, for arbitrary $N \in \N$, where it possibly exhibits simple poles at $z = \chi_n$, for $0 \leq n \leq N-1$. Moreover, $\mathcal{M}_{\varphi, \mathcal{P}_r }\curbr{a}(z) = \LandauO(|z|^{-1})$, as $\Im z\rightarrow\pm\infty$. Fix $N \in \N$ and $\lambda > \lambda_0$, such that $\zeta_\lambda > \Re\chi_N$. Then, for $K > 0$ and $\Re\chi_{N-1} < x_N < \Re\chi_N$, consider in the complex $z$-plane a rectangle with vertices $\curbr{x_0 - iK, x_N - iK, x_N + iK, x_0 + iK}$. The only possible singularities in its interior are the points $\chi_n$, for $0 \leq n \leq N-1$, so that the integral of $\Phi(z,\lambda)\mathcal{M}_{\varphi, \mathcal{P}_r }\curbr{a}(z)$ along the clockwisely traversed boundary of this rectangle, due to the residue theorem, exactly coincides with the sum of the associated residues, except for a negative sign. Because $\Phi(z, \lambda) \mathcal{M}_{\varphi, \mathcal{P}_r }\curbr{a}(z) = \LandauO(|z|^{-3})$, as $\Im z \rightarrow \pm\infty$, the integrals along the upper and lower edges eventually both vanish, as $K \rightarrow \infty$. Altogether, we get
\begin{align*}
\mathcal{L}_{\phi, \mathcal{P}_r }\curbr{a}(\lambda) = - \suml_{n = 0}^{N-1} \Res{ z = \chi_n }{ \Phi(z, \lambda) \mathcal{M}_{\varphi, \mathcal{P}_r }\curbr{a}(z) } + \frac{1}{2\pi i} \intl_{x_N-i\infty}^{x_N+i\infty} \Phi(z, \lambda) \mathcal{M}_{\varphi, \mathcal{P}_r }\curbr{a}(z) dz.
\end{align*}
Since the singularities at $z=\chi_n$ are simple poles, (\ref{ApplicOfExp5}) yields
\begin{align*}
\Res{ z = \chi_n }{ \Phi(z, \lambda) \mathcal{M}_{\varphi, \mathcal{P}_r }\curbr{a}(z) } = - e^{i\omega_\mathcal{P}} \Phi(\chi_n, \lambda) \operatorname{c}_\mathcal{P}(n) \hspace{1cm} (0 \leq n \leq N-1).
\end{align*}
With regard to the remainder integral, due to the choice of $x_N$, from Theorem \ref{TheoAnContAnIng}, we know that $\mathcal{M}_{\varphi, \mathcal{P}_r }\curbr{a}(z)$ is analytic in a neighborhood of the line $\Re z = x_N$, i.e., $\mathcal{M}_{\varphi, \mathcal{P}_r }\curbr{a}(x_N+iy)$ is continuous with respect to $y\in\R$. In addition, $\mathcal{M}_{\varphi, \mathcal{P}_r }\curbr{a}(x_N+iy) = \LandauO\curbr{|x_N+iy|^{-1}}$, as $y \rightarrow \pm\infty$. Hence, we infer the existence of $M > 0$, such that
\begin{align*}
\abs{ \intl_{x_N-i\infty}^{x_N+i\infty} \Phi(z, \lambda) \mathcal{M}_{\varphi, \mathcal{P}_r }\curbr{a}(z) dz } \leq \frac{M}{2\pi} \intl_{-\infty}^\infty \frac{ \abs{ \Phi(x_N+iy, \lambda) } }{ |x_N+iy| } dy =\bigO{ \Phi(x_N, \lambda) },
\end{align*}
where the big-$\LandauO$ estimate holds by Assumption \ref{LemExpwrtAsympScaleA}. From this assumption, for each $1\leq n \leq N$, especially since $\Re\chi_{n-1} < \Re\chi_n$, we also know that $\Phi(\chi_n,\lambda) = o\rrb{ \Phi(\chi_{n-1},\lambda) }$, as $\lambda \rightarrow \infty$. The above sum of residues therefore exhibits an asymptotic character for large $\lambda$. Finally, $\mathcal{L}_{ \phi, \mathcal{P} \setminus \mathcal{P}_r }\curbr{a}(\lambda) = \LandauO( \exp\curbr{-\lambda \inf_{t \in \mathcal{P} \setminus \mathcal{P}_r} \phi(t)} )$, as $\lambda \rightarrow \infty$. Recalling that $\phi(t) > \phi(t_0)$, for each $t \in \mathcal{P}$, by Assumption \ref{AssIngFctLT}, we conclude from Assumption \ref{LemExpwrtAsympScaleA} that the contribution of $\mathcal{L}_{ \phi, \mathcal{P} \setminus \mathcal{P}_r }\curbr{a}(\lambda)$ is negligible compared with $\mathcal{L}_{ \phi, \mathcal{P}_r }\curbr{a}(\lambda)$, which completes the proof.
\end{proof}

Given a function $\varphi : \mathcal{P} \rightarrow (0, \infty)$ that satisfies Assumption \ref{LemExpwrtAsympScaleA}, the previous theorem immediately yields a concise criterion for the existence of an asymptotic expansion for $\mathcal{L}_{ \phi, \mathcal{P} }\curbr{a}(\lambda)$ in terms of $\Phi(\cdot, \lambda)$. We proceed with various examples for such functions, where the first leads us to the well-known approximation due to Laplace.

\begin{cor}[Laplace] \label{CorExpPwrsLambd1}
Under Assumptions \ref{AnContTeixExpA} and \ref{AssIngFctLT}, define $\varphi(t) := \phi(t)-\phi(t_0)$, for $t \in \mathcal{P}$. If then also Assumption \ref{AssIngGenFct2} holds, we have
\begin{align*}
\mathcal{L}_{ \phi, \mathcal{P} }\curbr{a}(\lambda) \sim e^{i\omega_\mathcal{P} - \lambda \phi(t_0)} \suml_{n=0}^\infty \lambda^{-\chi_n} \Gamma\rb{\chi_n} \operatorname{c}_\mathcal{P}(n) \hspace{1cm} (\lambda \rightarrow \infty).
\end{align*}
\end{cor}

\begin{proof}
First notice that $\varphi(t) > 0$, for $t \in \mathcal{P}$, by Assumption \ref{AssIngFctLT}, as well as that $\mathcal{L}_{ \phi, \mathcal{P} }\curbr{a}(\lambda) = \exp\curbr{-\lambda \phi(t_0)} \mathcal{L}_{ \varphi, \mathcal{P} }\curbr{a}(\lambda)$, for all $\lambda \geq 0$. Now, according to the Cahen-Mellin integral (\ref{AsymptExpByAnalyticCont12}), for each $t \in \mathcal{P}$ and $\lambda, x_0 > 0$, we can write
\begin{align*}
e^{-\lambda\varphi(t)} = \frac{1}{2\pi i} \intl_{x_0-i\infty}^{x_0+i\infty} \curbr{\varphi(t)}^{-z} \lambda^{-z} \Gamma(z) dz .
\end{align*}
Clearly, $\Phi(z, \lambda) := \lambda^{-z} \Gamma(z)$ is holomorphic in $\Re z > 0$, and it follows from the exponential decay of the gamma function that $\Phi(z,\lambda) = \LandauO( \lambda^{-\Re z} |z|^{-2} )$, as $\Im z\rightarrow \pm \infty$, for any $\lambda > 0$ (see (5.11.7) in \cite{olver2010nist}). Thereby, the validity of Assumption \ref{LemExpwrtAsympScaleA} becomes obvious, so that the asserted expansion is easily obtainable by application of Theorem \ref{TheoExpwrtAsympScale} to $\mathcal{L}_{\varphi, \mathcal{P} }\curbr{a}(\lambda)$.
\end{proof}

The Laplacian approximation corresponds to an ordinary power series expansion in the sense of Poincaré. An alternative, where the asymptotic sequence exhibits a similar character, is established by an inverse factorial expansion.

\begin{cor}[inverse factorial expansion] \label{CorInvFacExpILam}
Under Assumptions \ref{AnContTeixExpA} and \ref{AssIngFctLT}, define $\varphi(t) := 1-\exp\curbr{-\phi(t)}$, for $t \in \mathcal{P}$. If then also Assumption \ref{AssIngGenFct2} holds, we have
\begin{align*}
\mathcal{L}_{\phi, \mathcal{P} }\curbr{a}(\lambda) \sim e^{i \omega_\mathcal{P} } \suml_{n=0}^\infty \frac{\Gamma(\lambda+1)\Gamma\rb{\chi_n} }{\Gamma\rb{\lambda+1+\chi_n}} \operatorname{c}_\mathcal{P}(n) \hspace{1cm} (\lambda \rightarrow \infty).
\end{align*}
\end{cor}

It is well-known that the $n$-th summand fulfills $\sim \mbox{const} \times \lambda^{-\chi_n}$, as $\lambda \rightarrow \infty$ (see (5.11.12) in \cite{olver2010nist}), i.e., the series indeed also exhibits a descending algebraic character with respect to $\lambda$. Furthermore, due to the series expansion of the exponential function, $\varphi$ and $\phi$ have the same order, whenever $\phi$ is small. The choice of Corollary \ref{CorInvFacExpILam} over Corollary \ref{CorExpPwrsLambd1} is hence mainly governed by the computability of the associated coefficients.

\begin{proof}[Proof of Corollary \ref{CorInvFacExpILam}]
In view of Assumption \ref{AssIngFctLT}, it is clear that $0 < \varphi(t) \leq 1$, for $t \in \mathcal{P}$. Furthermore, for fixed $\lambda > 0$, the beta function $\Phi(z, \lambda) := \Beta(\lambda+1, z)$ constitutes the Mellin transform of the continuous function $f_\lambda : (0, \infty) \rightarrow [0, 1)$, defined by $f_\lambda(s) := (1-s)^\lambda$, for $0 < s \leq 1$, and $f_\lambda(s) := 0$ else. Observe that $f_\lambda(\varphi(t)) = \exp\curbr{-\lambda\phi(t)}$. Thus, appealing to the Mellin inversion formula and the connection between the beta and the gamma function (cf. (5.12.1) in \cite{olver2010nist}), for each $t \in \mathcal{P}$ and $x_0 > 0$, it holds that
\begin{align} \label{AsymptExpByAnalyticCont7}
e^{-\lambda\phi(t)} = \frac{1}{2\pi i} \intl_{x_0-i\infty}^{x_0+i\infty} \rrb{ \varphi(t) }^{-z} \frac{\Gamma(\lambda+1)\Gamma(z)}{\Gamma(\lambda+1+z)} dz.
\end{align}
In addition, $\Phi(z, \lambda) \sim \lambda^{-z} \Gamma(z)$, as $\lambda\rightarrow\infty$, for each $z \in \C\setminus -\N_0$, and conversely $\Phi(z,\lambda) = \LandauO(|z|^{-\lambda-1})$, as $\abs{z}\rightarrow\infty$, for any $\lambda>0$ (cf. (5.11.12) in \cite{olver2010nist}). Finally, the required order estimate for the integral of $\Phi(z,\lambda)$ is a direct consequence of Lemma A.2 in \cite{Kaiser04102025}. Altogether, Assumption \ref{LemExpwrtAsympScaleA} holds and Theorem \ref{TheoExpwrtAsympScale} applies.
\end{proof}

So far, we only mentioned cases, in which an asymptotic series for $\mathcal{L}_{\phi, \mathcal{P} }\curbr{a}(\lambda)$, as $\lambda \rightarrow \infty$, in powers of $\lambda$ exists. However, this clearly need not be the case, e.g., if $\phi(t)$ exhibits exponential behaviour near the critical point, whereas $a(t)$ is algebraic there (see Example \ref{Ex2026042101} below). Then, Corollaries \ref{CorExpPwrsLambd1} and \ref{CorInvFacExpILam} are inapplicable. With regard to such situations, we invoke the recently introduced Mellin transform of the beta function, whose integral definition is
\begin{align} \label{2026061602}
\mathcal{M}_{\operatorname{B}}(z,\lambda-1) := \intl_0^\infty t^{-z-1} (1-e^{-t})^\lambda dt.
\end{align}
It represents a holomorphic function in $0 < \Re z < \lambda$, for each $\lambda > 0$. Partial integration shows that it is a special case of the general definition from \cite{Kaiser04102025}, namely $\mathcal{M}_{\operatorname{B}}(\cdot,\lambda-1) = \mathcal{M}_{\operatorname{B}}(\cdot,\lambda-1, 1)$. The main result of \cite{Kaiser04102025} is Theorem 3.3, according to which
\begin{align} \label{2026042102}
\mathcal{M}_{\operatorname{B}}(z,\lambda-1) \sim \frac{ (\psi(\lambda+1)+\gamma)^{-z} }{z} \hspace{1cm} (\lambda \rightarrow\infty),
\end{align}
for $z \in \C \setminus\curbr{0}$, where $\psi$ is the digamma function, with $\psi(\lambda+1) \sim \log \lambda$, while $\gamma \approx 0.5772$ refers to the Euler-Mascheroni constant. Hence, $\mathcal{M}_{\operatorname{B}}(z,\lambda-1)$ has asymptotic scaling properties, with a logarithmic character.

\begin{cor}[logarithmic-type expansion] \label{CorAnLogPwrFct}
Under Assumptions \ref{AnContTeixExpA} and \ref{AssIngFctLT}, define $\varphi(t) := (-\log(1-\exp\curbr{-\phi(t)}))^{-1}$, for $t \in \mathcal{P}$. If then also Assumption \ref{AssIngGenFct2} holds, we have
\begin{align*}
\mathcal{L}_{\phi, \mathcal{P} }\curbr{a}(\lambda) \sim e^{i \omega_\mathcal{P} } \suml_{n=0}^\infty \mathcal{M}_{\operatorname{B}}\rb{ \chi_n,\lambda-1} \operatorname{c}_\mathcal{P}(n) \hspace{1cm} (\lambda \rightarrow \infty).
\end{align*}
\end{cor}

Observe that $\varphi(t) \rightarrow 0$, if $\phi(t) \rightarrow 0$, yet, the decay is of much lower order than any power of $\phi(t)$, due to the logarithm.

\begin{proof}[Proof of Corollary \ref{CorAnLogPwrFct}]
By Assumption \ref{AssIngFctLT}, $\varphi(t) > 0$, for $t \in \mathcal{P}$. Hence, according to (\ref{2026061602}) and the Mellin inversion formula, for any $t \in \mathcal{P}$ and $0 < x_0 < \lambda$, we have
\begin{align*}
e^{-\lambda\phi(t)} = \frac{1}{2\pi i} \intl_{ x_0 - i\infty }^{ x_0 + i\infty} \rrb{ \varphi(t) }^{-z} \mathcal{M}_{\operatorname{B}}(z,\lambda-1) dz.
\end{align*}
Consider the function $\Phi(z, \lambda) := \mathcal{M}_{\operatorname{B}}(z,\lambda-1)$. Firstly, we know from Theorem 4.3 in \cite{Kaiser04102025} that $\mathcal{M}_{\operatorname{B}}(z,\lambda-1) = \LandauO( |z|^{-2} )$, as $\Im z \rightarrow \pm\infty$ in $0 < \Re z < \lambda$, for any fixed $\lambda>0$. By means of this theorem, for $\lambda>0$ and fixed $0 < x_0 < \lambda$, one can also show that
\begin{align*}
\intl_{-\infty}^\infty \frac{ |\mathcal{M}_{\operatorname{B}}(x_0+iy, \lambda-1)| }{ |x_0+iy| } dy \leq 2x_0 \mathcal{M}_{\operatorname{B}}(x_0, \lambda-1) \intl_0^\infty \frac{ dy }{ x_0^2+y^2 }.
\end{align*}
With the approximation (\ref{2026042102}) in mind, it finally becomes obvious that $\Phi(z, \lambda) = \mathcal{M}_{\operatorname{B}}(z,\lambda-1)$ as well satisfies Assumption \ref{LemExpwrtAsympScaleA}, implying that the asserted expansion is a direct consequence of Theorem \ref{TheoExpwrtAsympScale}.
\end{proof}

The case, in which $\mathcal{L}_{\phi, \mathcal{P} }\curbr{a}(\lambda)$ exhibits an exponential character is quite diffe\-rent from the previously considered cases, essentially because of the ambiguity of this notion. Generally speaking, if $\Phi(\cdot, \lambda)$ is an asymptotic scale, exponential character of $\mathcal{L}_{\phi, \mathcal{P} }\curbr{a}(\lambda)$ with respect to $\Phi(\cdot, \lambda)$ refers to an asymptotically dominating term of the form $\exp\curbr{z_1\Phi(z_2, \lambda)}$, for $z_1, z_2 \in \C$. Notice that $\exp\curbr{z_1\Phi(z_2, \lambda)}$ again gene\-rates an asymptotic scale with respect to each argument, given a fixed value of the other argument. In any case, exponential behaviour obviously depends on the asymptotic scale under consideration. Such a behaviour is particularly suggested, if the Mellin-Barnes representation of $\mathcal{L}_{\phi, \mathcal{P} }\curbr{a}(\lambda)$ (see the proof of Theorem \ref{TheoExpwrtAsympScale}) is applicable for all $x_0 > 0$, e.g., if $\curbr{\varphi(t)}^{-x_0}a(t)$ is uniformly bounded along $\mathcal{P}$. In Lemma \ref{Lem20260417}, we confirmed this hypothesis for $\Phi(z,\lambda) := \lambda^{-z} \Gamma(z)$. It is also noteworthy that the order of $\mathcal{L}_{\phi, \mathcal{P} }\curbr{a}(\lambda)$, according to Lemma \ref{Lem20260527}, can not exceed the factor $e^{-\lambda\sigma}$, for $\sigma > 0$. The actual challenge about exponential behaviour is the exact characterization. In the theory of Mellin-Barnes integrals, exponential expansions are often obtained from saddle point approximations or inverse factorial expansions (cf. $\S\S$5.3 and 5.4 in \cite{ParKam2001}). We do not outline these techniques right here and instead briefly mention an approach that follows the strategy of this text. Possibly the simplest case, in which $\mathcal{L}_{\phi, \mathcal{P} }\curbr{a}(\lambda)$ exhibits exponential behaviour (with respect to $\lambda$), occurs when $\phi(t_0)>0$, and this situation is already covered by the Laplacian approximation (Corollary \ref{CorExpPwrsLambd1}). On the other hand, cases in which $\mathcal{L}_{\phi, \mathcal{P} }\curbr{a}(\lambda)$ is exponential although $\phi(t_0)=0$ are manifold and suggest a sufficiently fast decay of $a(t)$ near $t=t_0$. Below, we confine to a specific situation. The corresponding result involves the confluent hypergeometric function of the second kind, also known as Tricomi's function (see (13.4.4) in \cite{olver2010nist}), for $\Re z, \Re b > 0$ and $c\in \C$, defined as
\begin{align} \label{2026061601}
U(b, c, z) := \frac{1}{\Gamma(b)} \intl_0^\infty e^{- zt} t^{b-1} (1+t)^{c-b-1} dt.
\end{align}
In $\S$10.3.2 in \cite{Temme2015}, it was shown that it can be expanded in terms of the modified Bessel function and hence, in view of (10.25.3) in \cite{olver2010nist}, for fixed $\Re z > 0$, we have
\begin{align} \label{2026042001}
\Gamma(\lambda+1)U(\lambda+1, 2, z) \sim \sqrt{\pi} z^{-\frac{3}{4}} (\lambda+1)^\frac{1}{4} e^{ \frac{z}{2} - 2\sqrt{z(\lambda+1)} } \hspace{1cm} (\lambda \rightarrow \infty).
\end{align}
Notice that this approximation has been obtained from a saddle point approach. In any case, it essentially justifies the final result of this section.

\begin{cor}[exponential-type expansion]
Under Assumptions \ref{AnContTeixExpA} and \ref{AssIngFctLT}, define $\varphi(t) := \exp\curbr{1-(1-\exp\curbr{-\phi(t)})^{-1}}$, for $t \in \mathcal{P}$. If then also Assumption \ref{AssIngGenFct2} holds, we have
\begin{align*}
\mathcal{L}_{\phi, \mathcal{P} }\curbr{a}(\lambda) \sim e^{i \omega_\mathcal{P} } \Gamma(\lambda+1) \suml_{n=0}^\infty U(\lambda+1, 2, \chi_n) \operatorname{c}_\mathcal{P}(n) \hspace{1cm} (\lambda \rightarrow \infty).
\end{align*}
\end{cor}

\begin{proof}
The function $f_\lambda : (0, \infty) \rightarrow [0, 1)$, defined by $f_\lambda(s) := (1-(1-\log(s))^{-1})^\lambda$, for $0 < s \leq 1$, and $f_\lambda(s) = 0$ else, is continuous, for any $\lambda > 0$. The associated Mellin transform has the strip of analyticity $\Re z > 0$, and a comparison with (\ref{2026061601}) shows that it satisfies the identity $\int_0^\infty s^{z-1} f_\lambda(s) ds = \Gamma(\lambda+1) U(\lambda+1, 2, z)$. Also, $0 < \varphi(t) \leq 1$, for $t \in \mathcal{P}$, by Assumption \ref{AssIngFctLT}, so that $f_\lambda(\varphi(t)) = \exp\curbr{-\lambda \phi(t)}$. Thus, the Mellin inversion formula, for any $t \in \mathcal{P}$ and $x_0>0$, implies that
\begin{align*}
e^{-\lambda \phi(t)} = \frac{1}{2\pi i} \intl_{x_0-i\infty}^{x_0+i\infty} \curbr{\varphi(t)}^{-z} \Gamma(\lambda+1)U(\lambda+1, 2, z) dz.
\end{align*}
Moreover, from (13.7.3) in \cite{olver2010nist}, we know that $U(\lambda+1, 2, z) \sim z^{-\lambda-1}$, as $\Im z \rightarrow \pm\infty$. Lastly, via partial integration, for fixed $\lambda, \Re z > 0$, we get
\begin{align*}
\Gamma(\lambda+1) U(\lambda+1, 2, z) = \intl_0^\infty e^{-zt} \frac{t^\lambda}{(1+t)^\lambda} dt = \frac{\lambda}{z} \intl_0^\infty e^{-zt} \frac{t^{\lambda-1}}{(1+t)^{\lambda+1}} dt.
\end{align*}
Thereof, by application of the triangle inequality and another subsequent integration by parts, it follows that
\begin{align*}
|U(\lambda+1, 2, z)| \leq \frac{\Re z}{|z|} U(\lambda+1, 2, \Re z).
\end{align*}
Altogether, in view of (\ref{2026042001}), we conclude that Assumption \ref{LemExpwrtAsympScaleA} is fulfilled, with $\Phi(z, \lambda) := \Gamma(\lambda+1)U(\lambda+1,2,z)$, so that Theorem \ref{TheoExpwrtAsympScale} is applicable.
\end{proof}

\section{Applications} \label{SecApp}

We finally discuss a few concrete examples to illustrate the applicability of our results. Each of the integrals that we consider here represents a bias between the so-called deconvolution function $\mathfrak{D}(\cdot, \lambda-1)$ (or its density $\mathfrak{d}(\cdot, \lambda-1)$; see \cite{Kaiser2025deconvolutionarbitrarydistributionfunctions}) and the probability distribution function $F_X$ (or density $f_X$) of a target random variable $X$, that is blurred by an additive random error $\varepsilon$. In \cite{Kaiser2025deconvolutionarbitrarydistributionfunctions}, by means of a representation as a Fourier-type integral, it was shown that $\mathfrak{D}(\cdot, \lambda-1) \rightarrow F_X$ and $\mathfrak{d}(\cdot, \lambda-1) \rightarrow f_X$, as $\lambda \rightarrow \infty$, under mild conditions. The subsequent examples quantify this rate of convergence. Roughly speaking, each bias integral consists of a kernel of the form $(1-h(t))^\lambda$, where $h : \R \rightarrow [0, 1]$ corresponds to the characteristic function of $\varepsilon$ (i.e., the Fourier-Stieltjes transform of the associated probability measure), and an amplitude function that is determined by the distribution of the target. We particularly focus on the evaluation of integrals that are outside the scope of Laplace's method.

\begin{ex} \label{Ex2026042101}
In our first example, we deal with an integral with a logarithmic asymptotic behaviour. Assume continuity of $f : (0, \infty) \rightarrow \C$, with $f(t) = \LandauO(t^{-\kappa})$, as $t \downarrow 0$, and $f(t) = f_0 t^{-\alpha}$, as $t \rightarrow \infty$, for $\kappa \geq 0$, $\alpha \in \N\setminus\curbr{1}$ and $f_0 \in \C\setminus\curbr{0}$. Also suppose the existence of $r_g > 0$ and a function $g(t)$ that is analytic in $D_{r_g}(0)$, with $g(0) = 0$ and $g(t) = f(\frac{1}{t})$, for $0 < t < r_g$. Finally, for fixed $\beta \geq 0$, $K \in \N_0$, with $\beta+K>0$, and $(\sigma_k)_{0 \leq k \leq K} \subset\R$, let
\begin{align*}
p(t) := t^\beta \suml_{k=0}^K \sigma_k t^k \hspace{1cm} (t >0)
\end{align*}
be the product of a power function and a non-trivial polynomial, such that $p(t) \geq 0$, for $t \geq 0$. In this event, for all sufficiently large $\lambda > 0$, we have absolute convergence of the integral
\begin{align*}
\Epsilon(\lambda) := \intl_0^\infty (1 - e^{-p(t)})^\lambda f(t) dt.
\end{align*}
If $\beta+K \leq 2$ and $\sigma_k > 0$, for all $0 \leq k \leq K$, then $\exp\curbr{-p(t)}$ represents the product of characteristic functions of certain stable distributions. Our aim is a full expansion for the above integral, as $\lambda \rightarrow\infty$. For large values of $\lambda$, it has a decreasing character, and the main contribution comes from the point at infinity. Hence, we first map this infinite to a finite critical point. Thereof, with $\phi(s) := -\log\curbr{1-\exp\curbr{-p(\frac{1}{s})}}$, we get
\begin{align*}
\Epsilon(\lambda) = \intl_0^\infty e^{-\lambda \phi(s)} s^{-2} f\rb{\frac{1}{s}} ds.
\end{align*}
In this, the phase function $\phi(s)$ vanishes exponentially fast, while the amplitude is algebraic, as $s\downarrow0$. The Laplacian approximation (Corollary \ref{CorExpPwrsLambd1}) is thus inapplicable. On the other hand,
\begin{align*}
\varphi(s) := \rrb{ p\rb{\frac{1}{s}} }^{-1} = s^{\beta+K} \rrb{\suml_{k=0}^K \sigma_k s^{K-k} }^{-1}
\end{align*}
exhibits algebraic behaviour at $s=0$. The last function is exactly the product of a power and an analytic function, with $\varphi(s) \sim \sigma_K^{-1} s^{\beta+K}$, as $s \downarrow 0$. Furthermore, by assumption, $a(s) := s^{-2}g(s)$ is analytic in $D_{r_g}(0)$, with $a(s) = s^{-2}f\rb{\frac{1}{s}}$, for $0 < s < r_g$, and $a(s) \sim f_0s^{\alpha-2}$, as $s \downarrow 0$. In summary, we conclude the applicability of Corollary \ref{CorAnLogPwrFct}, leading to the expansion
\begin{align} \label{ExNormalCauchyConv4}
\Epsilon(\lambda) \sim \suml_{n=0}^\infty \mathcal{M}_{\operatorname{B}} \rb{ \frac{\alpha-1+n}{\beta+K},\lambda-1 } \operatorname{c}(n) \hspace{1cm} (\lambda\rightarrow\infty),
\end{align}
where the coefficients, according to (\ref{ZeroAlgBound21b}), are given by
\begin{align*}
\operatorname{c}(n) := \frac{1}{\beta + K} \frac{1}{n! } \frac{d^n }{ds^n } \frac{ f(\frac{1}{s}) }{ s^\alpha } \rb{ \suml_{k=0}^K \sigma_k s^{K-k} }^\frac{\alpha-1+n}{\beta+K} \Bigg|_{s=0} \hspace{1cm} (n\in\N_0).
\end{align*}
Figure \ref{exp_phase_plots.fig} exemplarily shows the absolute deviation of various partial sums of the above expansion from the full integral $\Epsilon(\lambda)$, on a logarithmic scale. In the considered situation, an exponentially distributed target is blurred by an error variable that consists of a convolution of a Cauchy and a Gauss distribution.
\end{ex}

\begin{figure}
\centering
\resizebox*{\textwidth}{!}{\includegraphics{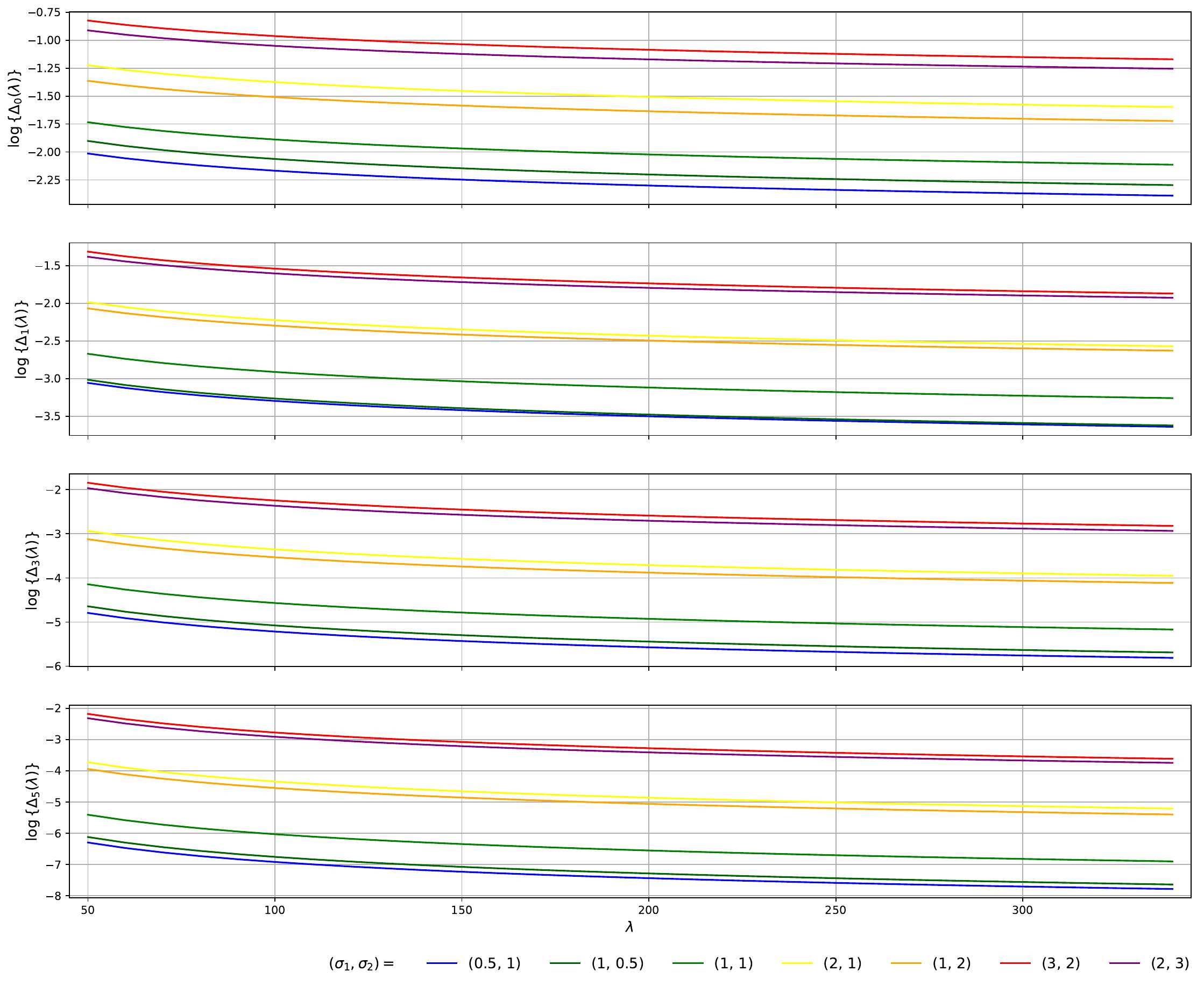}}
\vspace{-0.5cm}
\caption{Logarithmized absolute deviation $\log\curbr{ \Delta_N(\lambda)}$ of the integral $\Epsilon(\lambda)$ from the $N$-th partial sum of the asymptotic series (\ref{ExNormalCauchyConv4}), if $p(t) := \sigma_1t+\sigma_2t^2$ and $f(t) := \curbr{t(1-i t)}^{-1}$. In each row, $N \in \curbr{0, 1, 3, 5}$ is fixed and graphs correspond to distinct values of $\sigma_1, \sigma_2 > 0$. Observe the remarkable improvement in accuracy between the first and the second order approximation. The reason is that the amplitude $f(t)$ and hence also $\Epsilon(\lambda)$ have non-zero real and imaginary parts. In this, neither the dominating term of $\Re \Epsilon(\lambda)$ nor of $\Im \Epsilon(\lambda)$ is negligible. Therefore, the purely imaginary first order approximation $\Epsilon(\lambda) \sim \frac{i}{2} \sqrt{\sigma_2} \mathcal{M}_{\operatorname{B}} \rb{ \frac{1}{2},\lambda-1 }$ is highly inaccurate.} \label{exp_phase_plots.fig}
\end{figure}

Our second and last example is special, since a direct application of our results is not possible. The dominating term of the integral under consideration has a very unique structure, due to the interrelation between an oscillatory and an exponential factor in the kernel, which incurs an infinite sequence of critical points. A slightly similar setup was treated in \cite{Paris2020SineInts}.

\begin{ex}\label{CosExpExample}
Consider a situation, in which a Cauchy-type target random variable (with scale $\theta > 0$) is blurred by an additive error variable that is a convolution of a Cauchy-type (with scale $\sigma > 0$) and a discrete uniform distribution (on $\curbr{-1, 1}$). The bias between the deconvolution function and the distribution function of the target variable, for $\lambda > 0$, is then given by
\begin{align} \label{CauchyDiscrUnifEps1}
\Ypsilon_{\sigma, \theta}(\lambda) := \intl_0^\infty \rrb{ 1-e^{-\sigma t} \cos^2(t) }^\lambda e^{-\theta t} dt.
\end{align}
We want to study the dominating behaviour as $\lambda \rightarrow \infty$. Obviously, the kernel of $\Ypsilon_{\sigma, \theta}(\lambda)$ has an infinite sequence of local maxima at the zeros of the cosine and an additional maximum at infinity, where it exhibits exponential behaviour. By analy\-ti\-city of the ingredient functions in a neighborhood of $\pi(\frac{1}{2} + k)$, for any $k \in \N_0$, the asymptotic behaviour of the integral along each finite segment of the positive real axis can be approximated through the method of Laplace (Corollary \ref{CorExpPwrsLambd1}). Yet, this is no longer applicable when an arbitrary infinite segment of the positive real axis is considered, because the zeros of the kernel lie dense in any neighborhood of infinity, i.e., infinity is not an isolated critical point. At the same time, as we will see below, it is important to distinguish between the contribution from the finite zeros and the zero at infinity. In particular, the asymptotic behaviour of $\Ypsilon_{\sigma, \theta}(\lambda)$ need not coincide with the sum of the contributions from the points $\pi(\frac{1}{2} + \N_0)$. To begin, we represent the kernel in terms of the binomial integral (\ref{AsymptExpByAnalyticCont7}), because in that way the oscillatory can be separated from the exponential factor, so that the effect of each factor becomes obvious. Similar to Lemma \ref{LemAnGenFctAlgKer}, with $\arg\curbr{\cos^2(t)} = 0$, for $0 < t < \infty$, it is easy to verify that the generating function
\begin{align} \label{CauchyDiscrUnifEps2}
\mathfrak{M}_{\sigma, \theta}(z) := \intl_0^\infty \rrb{ \cos^2(t) }^{-z} e^{-(\theta-\sigma z) t} dt
\end{align}
converges absolutely for $\Re z < \min\rrb{\frac{1}{2}, \frac{\theta}{\sigma}}$ and represents a $z$-holomorphic function in this half plane. The boundary conditions are determined by the finite zeros and the exponential behaviour of the integrand at infinity. In accordance with these findings, for $0<x_0<\min\rrb{\frac{1}{2}, \frac{\theta}{\sigma}}$, with the aid of (\ref{AsymptExpByAnalyticCont7}), we obtain from (\ref{CauchyDiscrUnifEps1}) the Mellin-Barnes integral
\begin{align} \label{CauchyDiscrUnifEps4}
\Ypsilon_{\sigma, \theta}(\lambda) = \frac{1}{2 \pi i} \intl_{x_0-i\infty}^{x_0+i\infty} \frac{\Gamma(\lambda+1)\Gamma(z)}{\Gamma(\lambda+1+z)} \mathfrak{M}_{\sigma, \theta}(z) dz.
\end{align}
For fixed $\Re z < \min\rrb{\frac{1}{2}, \frac{\theta}{\sigma}}$, elementary manipulations of the integral definition of $\mathfrak{M}_{\sigma, \theta}(z)$ in (\ref{CauchyDiscrUnifEps2}) and a subsequent application of the geometric sum formula yield
\begin{align*}
\mathfrak{M}_{\sigma, \theta}(z) = \frac{ e^{(\theta-\sigma z) \frac{\pi}{2}} }{ 2\sinh \rrb{(\theta-\sigma z) \frac{\pi}{2} } } \intl_0^\pi \rrb{ \cos^2(t) }^{-z} e^{-(\theta-\sigma z) t} dt.
\end{align*}
Finally, defining
\begin{align} \label{CauchyDiscrUnifEps6}
\mathfrak{N}_{\sigma, \theta}(z) := \intl_0^\frac{\pi}{2} \rrb{ \sin(t) }^{-z} \cosh\rrb{ \rb{\theta-\frac{\sigma z}{2} } t} dt,
\end{align}
by separation of the integration path, we arrive at
\begin{align} \label{CauchyDiscrUnifEps5}
\mathfrak{M}_{\sigma, \theta}(z) = \frac{ \mathfrak{N}_{\sigma, \theta}(2z) }{ \sinh \rrb{(\theta-\sigma z) \frac{\pi}{2} } }.
\end{align}
In this, $\mathfrak{N}_{\sigma, \theta}(2z)$ is holomorphic in $\Re z < \frac{1}{2}$, by Lemma \ref{LemAnGenFctAlgKer}, whereas the reciprocal hyperbolic sine is meromorphic, with an infinite sequence of simple poles at $z_k := \frac{\theta-i2k}{\sigma}$, for $k\in\Z$. As $z \rightarrow z_k$, we have
\begin{align} \label{2026060201}
\frac{ 1 }{ \sinh \rrb{(\theta-\sigma z)\frac{\pi}{2} } } = \frac{ 2 (-1)^{k+1} }{\sigma\pi(z-z_k)} + \frac{\pi \sigma}{12} (-1)^k (z-z_k) + \bigO{ (z-z_k)^3 }.
\end{align}
Altogether, the whole representation (\ref{CauchyDiscrUnifEps5}) is valid in $\Re z < \frac{1}{2}$, where it represents an analytic function, unless $2\theta < \sigma$, in which circumstances it is meromorphic, with an infinite sequence of simple poles on the line $\Re z = \frac{ \theta }{ \sigma }$. In this last case, we already have an analytic continuation of the integral definition of $\mathfrak{M}_{\sigma, \theta}(z)$. Now, upon plugging (\ref{CauchyDiscrUnifEps5}) into (\ref{CauchyDiscrUnifEps4}), still for $0 < x_0 < \min\curbr{\frac{1}{2}, \frac{\theta}{\sigma}}$, we obtain
\begin{align} \label{2026060102}
\Ypsilon_{\sigma, \theta}(\lambda) = \frac{1}{2 \pi i} \intl_{x_0-i\infty}^{x_0+i\infty} \frac{\Gamma(\lambda+1)\Gamma(z)}{\Gamma(\lambda+1+z)} \frac{ \mathfrak{N}_{\sigma, \theta}(2z) }{ \sinh \rrb{(\theta-\sigma z) \frac{\pi}{2} } } dz.
\end{align}
If $2\theta < \sigma$, it is directly possible to extract the dominating term in the asymptotic expansion of $\Ypsilon_{\sigma, \theta}(\lambda)$, by collecting the residues of the simple poles at $(z_k)_{k \in \Z}$. To this end, notice that $\mathfrak{N}_{\sigma, \theta}(2z) = \LandauO(1)$ in $\Re z < \frac{1}{2}$ and that $\curbr{\sinh\curbr{(\theta-\sigma z)\frac{\pi}{2}}}^{-1} = \LandauO(1)$, as $\Im z \rightarrow \pm\infty$, except along the indicated sequence of simple poles. Hence, due to the beta function, the leading asymptotic behaviour of the integrand in (\ref{2026060102}) is $\LandauO(\abs{z}^{-\lambda-1})$, as $\Im z \rightarrow\pm\infty$, in $\Re z < \frac{1}{2}$. Consequently, similar to the proof of Theorem \ref{TheoExpwrtAsympScale}, it is admitted to displace the integration path to the right over the infinite sequence of poles at $( z_k )_{k \in \Z}$, to match a line $\frac{\theta}{\sigma} < x_1 < \frac{1}{2}$. Thereof, in view of (\ref{2026060201}) and by Lemma A.2 in \cite{Kaiser04102025}, we get
\begin{align*}
\Ypsilon_{\sigma, \theta}(\lambda) = \frac{2}{\sigma\pi} \suml_{k=-\infty}^\infty \frac{\Gamma(\lambda+1)\Gamma(z_k)}{\Gamma(\lambda+1+z_k)} \mathfrak{N}_{\sigma, \theta}(2z_k) (-1)^k + \LandauO(\lambda^{-x_1}) \hspace{1cm} (\lambda\rightarrow\infty).
\end{align*}
Observe that $z_{-k} = \overline{z_k}$, for all $k \in \Z$, and that $z_k$ lies in the range of validity of the integral definition of the beta function and of $\mathfrak{N}_{\sigma, \theta}(2z)$. Thus, for brevity, introducing
\begin{align*}
q_{\sigma, \theta}(\lambda) := 2 \suml_{k=1}^\infty \Re \frac{\Gamma\rb{\lambda+1+\frac{\theta}{\sigma}} \Gamma\rb{\frac{\theta}{\sigma} - \frac{ i 2k}{\sigma}} }{ \Gamma\rb{\lambda+1+\frac{\theta}{\sigma}-\frac{ i 2k}{\sigma}} } \mathfrak{N}_{\sigma, \theta}(2z_k) (-1)^k,
\end{align*}
we have $q_{\sigma, \theta}(\lambda) = \LandauO(1)$, as $\lambda \rightarrow \infty$, and still for $\sigma > 2\theta$ we conclude that
\begin{align} \label{2026050501}
\Ypsilon_{\sigma, \theta}(\lambda) \sim \frac{2}{\sigma\pi} \frac{\Gamma\rb{\lambda+1} }{ \Gamma\rb{\lambda+1+\frac{\theta}{\sigma}} } \rrb{ \Gamma\rb{\frac{\theta}{\sigma}} \mathfrak{N}_{\sigma, \theta}\rb{-\frac{2\theta}{\sigma}} + q_{\sigma, \theta}(\lambda) } \hspace{0.5cm} (\lambda\rightarrow\infty).
\end{align}
To determine the dominating term in the case $\sigma \leq 2\theta$ or subsequent terms in the above asymptotic expansion, we require an analytic continuation of $\mathfrak{N}_{\sigma, \theta}(z)$. For this, however, since the associated amplitude function depends on the additional variable $z$, we can not directly refer to Theorem \ref{TheoAnContAnIng}. Fortunately, we can nevertheless proceed in the fashion of $\S$\ref{AnContTeixExp}. The key is to first observe that $\curbr{\cos(t)}^{-1} \cosh\rrb{\rb{\theta - \frac{\sigma z}{2}}t}$, at $t=t_0$, can be expanded in powers of $\sin(t)$, similar to Lemma \ref{LemZeroAlgBound15}, essentially because no singularities are affected by the choice of $z$ and because the sine is already conformal at $t=0$. As a consequence, it exists $r_s > 0$, for which $\sin(t)$ is one-to-one in $D_{r_s}(0)$. Additionally let $r_c>0$ be the greatest radius, such that $\cos(t) \neq 0$ for all $t \in D_{r_c}(0)$. Then, analogous to Lemma \ref{LemZeroAlgBound15}, for fixed $0<r_0<\min\curbr{r_s, r_c}$ and $(t, N, z) \in D_{r_0}(t_0) \times \N_0 \times \C$, one can show that
\begin{align} \label{2026060101}
\frac{\cosh\rrb{\rb{\theta - \frac{\sigma z}{2}}t}}{\cos(t)} = \suml_{n=0}^{N-1} \curbr{\sin(t)}^n \operatorname{c}(n, z) + \operatorname{C}(t, N, z),
\end{align}
where the coefficients are given by
\begin{align*}
\operatorname{c}(n, z) := \frac{1}{2\pi i} \ointl_{\Gamma_{r_0}(0)} \frac{\cosh\rrb{\rb{\theta - \frac{\sigma z}{2}}w}}{ \rrb{\sin(w)}^{n+1} } dw \hspace{1cm} (0 \leq n \leq N-1),
\end{align*}
while the remainder integral is of the form
\begin{align*}
\operatorname{C}(t, N, z) := \curbr{\sin(t)}^N \frac{1}{2\pi i} \ointl_{\Gamma_{r_0}(0)} \frac{\cosh\rrb{\rb{\theta - \frac{\sigma z}{2}}w}}{ \rrb{\sin(w)}^N } \frac{ dw }{ \sin(w) - \sin(t)}.
\end{align*}
Due to the entireness of the hyperbolic cosine, we moreover compute
\begin{align*}
\operatorname{c}(n, z) = \suml_{k=0}^{ \lfloor \frac{n+1}{2} \rfloor } \frac{ \rb{\theta - \frac{\sigma z}{2}}^{2k} }{ (2k)! } \frac{d^{n-2k}}{dw^{n-2k}} \rrb{ \frac{w}{\sin(w)} }^{n+1} \Bigg|_{w=0}.
\end{align*}
The last representation especially shows that $\operatorname{c}(n, z) = \LandauO\curbr{ |z|^{n+1} }$, as $|z| \rightarrow \infty$. Ana\-lo\-gously, one verifies that $\operatorname{C}(t, N, z) = \LandauO\curbr{ |z|^{N+1} }$, as $|z| \rightarrow \infty$. Thus, with the aid of the expansion (\ref{2026060101}), as in Theorem \ref{TheoAnContAnIng}, one can verify that $\mathfrak{N}_{\sigma, \theta}(2z)$ admits a meromorphic continuation to the half plane $\Re z < \frac{N+1}{2}$, with a simple pole at $z = \frac{n+1}{2}$ and associated residue equal to $- \frac{1}{2} \operatorname{c}(n, 2z)$, for each $0 \leq n \leq N-1$. The main difference to the result from the theorem, due to the dependence of the amplitude on the argument $z$, is that $\mathfrak{N}_{\sigma, \theta}(2z) = \LandauO\curbr{ |z|^{N+1} }$, as $\Im z \rightarrow \pm\infty$. Nevertheless, in view of the asymptotic behaviour of the beta function, we conclude that a complete asymptotic expansion for $\Ypsilon_{\sigma, \theta}(\lambda)$ also can be obtained from a rightward displacement of the integration path of the Mellin-Barnes integral (\ref{2026060102}), across an arbitrary but finite number of poles of the integrand. According to our previous observations, these lie at $z = \frac{n+1}{2}$, for $0 \leq n \leq N-1$, and at $(z_k)_{k \in \Z}$ and are all simple, unless it exists $n_0 \in \N$ with $2\theta = (n_0+1)\sigma$, in which case the pole at $z = z_0 = \frac{n_0+1}{2}$ is of order two. This second order pole specifies the dominating behaviour if $\sigma = 2\theta$, in which circumstances one can verify that $\Ypsilon_{\sigma, \theta}(\lambda) \sim \mbox{const} \times \log(\lambda) \lambda^{-\frac{1}{2}}$, as $\lambda \rightarrow \infty$. Finally, if $\sigma < 2\theta$, one easily shows that $\Ypsilon_{\sigma, \theta}(\lambda) \sim \mbox{const} \times \lambda^{-\frac{1}{2}}$, as $\lambda \rightarrow \infty$. This last case is the only situation, where the leading asymptotic behaviour of $\Ypsilon_{\sigma, \theta}(\lambda)$ is completely determined by the zeros of the cosine in the initial representation (\ref{CauchyDiscrUnifEps1}). However, the exponential factors then play an additional role when it comes to higher order terms. Altogether, for a reliable approximation of an integral with an infinite number of critical points, the previous example has shown that it is not recommended to confine to a local study and to eventually sum up the single contributions. Figure \ref{cos_exp_phase_plots.fig} illustrates the asymptotic character of $\Ypsilon_{\sigma, \theta}(\lambda)$ for selected parametrizations.
\end{ex}

\begin{figure}
\centering
\resizebox*{\textwidth}{!}{\includegraphics{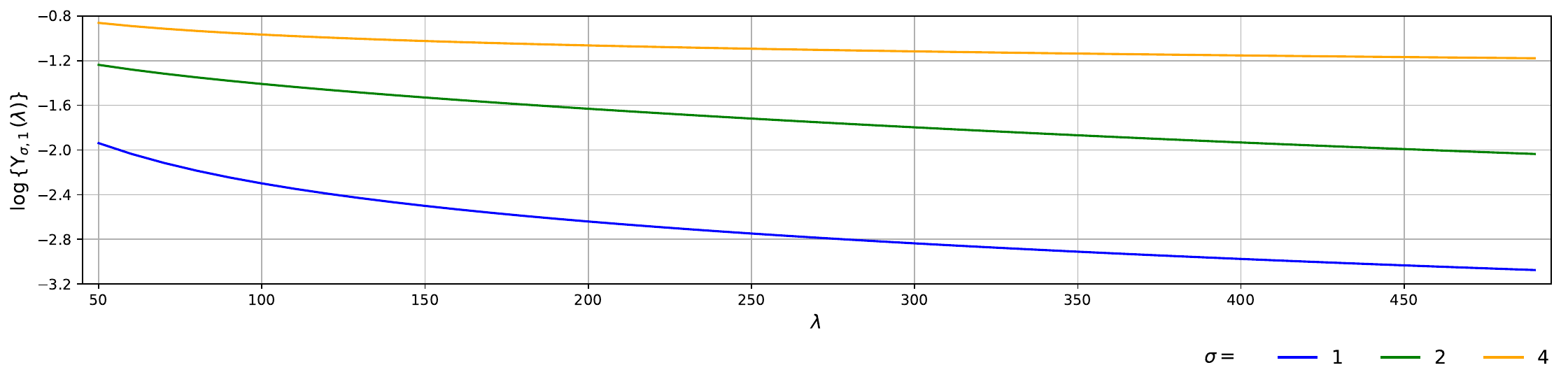}}
\vspace{-0.5cm}
\caption{Logarithmized graph for the integral $\Ypsilon_{\sigma, 1}(\lambda)$ from Example \ref{CosExpExample}, for various values of $\sigma$. In view of the obtained approximations, $\Ypsilon_{\sigma, 1}(\lambda) \sim \mbox{const} \times \lambda^{-\min\curbr{\frac{1}{2}, \frac{1}{\sigma}}}$, as $\lambda \rightarrow \infty$, except for $\sigma = 2$, where an additional logarithmically divergent factor appears.} \label{cos_exp_phase_plots.fig}
\end{figure}

\section{Conclusion} \label{SecConc}

In this text, we extended Laplace's method in an intuitive fashion with the aid of Mellin transforms, thereby facilitating the asymptotic evaluation of the generalized Laplace transform $\mathcal{L}_{\phi, \mathcal{P}}\curbr{a}(\lambda)$ for a wider class of ingredient functions. The technique essentially relies on an appropriate representation for the kernel of $\mathcal{L}_{\phi, \mathcal{P}}\curbr{a}(\lambda)$, through which the initial Laplacian integral becomes a Mellin-Barnes integral, with asymptotic scaling function $\Phi(z, \lambda)$ and generating function $\mathcal{M}_{\varphi, \mathcal{P}}\curbr{a}(z)$. The path of this Mellin-Barnes integral is a vertical line, whose intersection with the real axis indicates the asymptotic order of $\mathcal{L}_{\phi, \mathcal{P}}\curbr{a}(\lambda)$. If $\varphi(t)$ and $a(t)$ have power series expansions in a neighborhood of the critical points, we have shown that $\mathcal{M}_{\varphi, \mathcal{P}}\curbr{a}(z)$ can be continued to a meromorphic function in the complex plane, with an infinite sequence of equidistant simple poles with increasing real parts. An asymptotic expansion for $\mathcal{L}_{\phi, \mathcal{P}}\curbr{a}(\lambda)$ in terms of the given asymptotic scale $\Phi(z, \lambda)$ then can be obtained by collecting the residues of these poles. The asymptotic scales that we considered here have algebraic, logarithmic and exponential character. In a concrete example, the character of the applicable asymptotic scales is specifically pre-determined by our analyticity assumption (Assumption \ref{AssIngGenFct2}) and hence by the properties of $\phi(t)$. The reason is that, e.g., $\phi(t)$ and $-\log\phi(t)$ never exhibit the same local behaviour at a critical point. However, the applicant can choose between different types of scales with the same character, for instance, to obtain an inverse factorial instead of a power series expansion.
\\
\hspace*{1em}Finally, if an appropriate representation for the kernel is available, by reference to Theorem \ref{TheoExpwrtAsympScale}, it is easy to establish further expansions for $\mathcal{L}_{\phi, \mathcal{P}}\curbr{a}(\lambda)$. Moreover, the kernel actually need not be of exponential-type. For instance, consider the class of generalized Stieltjes transforms $\mathcal{S}_{\phi, \mathcal{P}}\curbr{a}(\lambda) := \int_\mathcal{P} \frac{ a(t) }{1 + \lambda\phi(t)} dt$, for suitable functions $a$ and $\phi>0$. The integral $\mathcal{S}_{\phi, \mathcal{P}}\curbr{a}(\lambda)$ is of Laplace-type, since the main contribution as $\lambda \rightarrow \infty$ comes from the infima of $\phi(t)$. Use of the inverse Mellin transform of $\frac{1}{1+t}$, for admissible $0 < x_0 < 1$, formally yields the Mellin-Barnes representation
\begin{align*}
\mathcal{S}_{\phi, \mathcal{P}}\curbr{a}(\lambda) = \frac{1}{2\pi i} \intl_{x_0-i\infty}^{x_0+i\infty} \lambda^{-z} \Gamma(z)\Gamma(1-z) \mathcal{M}_{\phi, \mathcal{P}}\curbr{a}(z) dz.
\end{align*}
If the analytic continuation of $\mathcal{M}_{\phi, \mathcal{P}}\curbr{a}(z)$ is available from Theorem \ref{TheoAnContAnIng}, an asymptotic expansion for $\mathcal{S}_{\phi, \mathcal{P}}\curbr{a}(\lambda)$ directly can be obtained by displacement of the integration path. Notice that, unlike in case of an exponential kernel, additional poles show up in the half plane $\Re z > 0$, which are inherited from the integral representation for the kernel.

\section*{Declarations}

\bmhead{Funding} This research did not receive funding.

\backmatter

\begin{appendices}

\section{Analyticity of a class of generating functions}

We here summarize a few basic analyticity properties of generating functions of the form
\begin{align*}
\mathcal{M}(z) := \intl_0^1 \rrb{\varphi(t)}^{-z} a(t,  z) dt,
\end{align*}
in case that each ingredient function is at most of algebraic order.

\begin{lem} \label{LemAnGenFctAlgKer}
Let $\varphi : (0, 1] \rightarrow (0, \infty)$ and $a : (0, 1] \times \C \rightarrow \C$ be continuous. Suppose that $0 \leq \lim_{t \downarrow 0} \varphi(t) < \infty$, as well as the existence of $\beta > 0$ and $\alpha > -1$, for which $\sup_{t \in (0, 1]} t^\beta \curbr{ \varphi(t) }^{-1} = \LandauO( 1 )$ and $\sup_{(t, z) \in (0, 1] \times \C,~ x_1 < \Re z < x_2} t^{-\alpha} |a(t, z)| = \LandauO( 1 )$, for all $x_1 < x_2$. Also assume entireness of $z \mapsto a(t, z)$, for each $t \in (0, 1]$. In these circumstances, $\mathcal{M}(z)$ is holomorphic in its region of absolute convergence, i.e., in the half plane
\begin{align*}
\Re z < \frac{\alpha+1}{\beta},
\end{align*}
and the associated derivatives can be obtained by differentiation under the sign of integration. In particular, if $z_0 \in \C$ with $\Re z_0 < \beta^{-1}(\alpha+1)$, and $f$ is analytic in a punctured neighborhood of $z_0$, then
\begin{align*}
\Res{z=z_0}{ f(z)\mathcal{M}(z) } = \intl_0^1 \Res{z=z_0}{ \rrb{\varphi(t)}^{-z} f(z) } a(t, z) dt.
\end{align*}
\end{lem}

Notice that the integral $\mathcal{M}(z)$ establishes an entire function, whenever $\beta > 0$ is arbitrary. The proof of the lemma uses various classical results from complex calculus.

\begin{proof}[Proof of Lemma \ref{LemAnGenFctAlgKer}]
In the sequel, for brevity, we write $B := \sup_{s \in (0, 1]} s^\beta \curbr{\varphi(s)}^{-1}$. Then, for fixed $0 \leq \Re z < \beta^{-1}(\alpha+1)$, we first of all get
\begin{align*}
\intl_0^1 \curbr{\varphi(t)}^{-\Re z} |a(t, z)| dt \leq B^{\Re z} \supl_{s \in (0, 1] } s^{-\alpha} |a(s, z)| \intl_0^1 t^{\alpha-\beta \Re z} dt < \infty.
\end{align*}
A similar bound can be obtained for fixed $\Re z \leq 0$, since $\sup_{s \in (0, 1]} \varphi(s) = \LandauO(1)$ and $\alpha>-1$. Thus, so far, we verified absolute convergence of $\mathcal{M}(z)$ for $\Re z < \beta^{-1}(\alpha+1)$. Regarding analyticity, in view of Theorem 5.6.1 in \cite{wegert2012visual} and our assumptions, we infer entireness of
\begin{align*}
\mathcal{M}_n(z) := \intl_\frac{1}{n}^1 \rrb{\varphi(t)}^{-z} a(t, z) dt \hspace{1cm} (n\in\N).
\end{align*}
Additionally consider a compact subset $E$ from the strip $0 < \Re z < \beta^{-1}(\alpha+1)$, with $x_- := \min\rrb{\Re z : z\in E}$ and $x_+ := \max\rrb{\Re z : z\in E}$. Then, $0 < x_- \leq x_+ < \beta^{-1}(\alpha+1)$, and we deduce, for any $z\in E$, that
\begin{align*}
\abs{\mathcal{M}_n(z) - \mathcal{M}(z)} \leq B^{\max\rrb{x_-, x_+}} \supl_{ \substack{ (s,z) \in (0, 1] \times \C \\ x_- < \Re z < x_+ } } s^{-\alpha} |a(s, z)| \frac{ n^{ \beta x_+ - 1-\alpha } }{ 1+\alpha-\beta x_+ }.
\end{align*}
It shows that $\mathcal{M}_n(z) \rightarrow \mathcal{M}(z)$, as $n \rightarrow \infty$, uniformly with respect to $z \in E$. Therefore, through Theorem 5.1.3 in \cite{wegert2012visual}, we eventually conclude analyticity of $\mathcal{M}(z)$ in $0 < \Re z < \beta^{-1}(\alpha+1)$. At the same time, for $k \in \N$, the $k$-th derivative of $\mathcal{M}_n(z)$ can be obtained from $k$-times differentiation under the integral sign, viz $\mathcal{M}_n^{(k)}(z) = \int_\frac{1}{n}^1 \curbr{\varphi(t)}^{-z} (-\log\curbr{ \varphi(t) })^k a(t, z) dt$, and $\mathcal{M}_n^{(k)}(z)$ converges to the $k$-th derivative of $\mathcal{M}(z)$, uniformly in any compact subset of the indicated strip, again according to Theorems 5.1.3 and 5.6.1 in \cite{wegert2012visual}. Consequently, also $\mathcal{M}(z)$ is differentiable under the sign of integration. By analogous arguments, one verifies these properties in $\Re z < 0$. Finally, $\mathcal{M}(z)$ and its derivatives are obviously continuous in the whole half plane $\Re z < \beta^{-1}(\alpha+1)$ and specifically along the line $\Re z = 0$, so that Morera's theorem (see Corollary 4.2.23 in \cite{wegert2012visual}) altogether implies analyticity there. The proof is completed, as the residue is exactly the Laurent coefficient of $f(z)\mathcal{M}(z)$ associated with the term $(z-z_0)^{-1}$, where the coefficients in the series expansion of $\mathcal{M}(z)$ result from differentiation.
\end{proof}

\section{Some estimates for a generalized Laplace transform with a large argument}

This paragraph provides various asymptotic estimates for the Laplace transform $\mathcal{L}_{ \phi, \mathcal{P} }\curbr{a}(\lambda)$ under broad assumptions on the ingredient functions. We will give a sufficient condition for exponential decay, as $\lambda \rightarrow \infty$, and it will turn out that $\mathcal{L}_{ \phi, \mathcal{P} }\curbr{a}(\lambda)$ can not vanish faster than $e^{-\lambda\sigma}$, where $\sigma > 0$ is determined through $\phi$.

\begin{lem} \label{Lem20260417}
Consider an arbitrary $\mathcal{P} \subseteq \R$ and $\phi, a : \mathcal{P} \rightarrow \C$, such that it exists $\varepsilon>0$, with $|\arg\curbr{\phi(t)}| < \frac{\pi}{2}-\varepsilon$, for $t \in \mathcal{P}$, and $\int_\mathcal{P} |a(t)| dt < \infty$. If there exists $x_0 > 0$, with $\int_\mathcal{P} |\phi(t)|^{-x_0} |a(t)| dt < \infty$, then
\begin{align*}
\mathcal{L}_{ \phi, \mathcal{P} }\curbr{a}(\lambda) = \LandauO\rb{ \lambda^{-x_0} } \hspace{1cm} (\lambda\rightarrow \infty).
\end{align*}
If $x_0>0$ is arbitrary, the decay of $\mathcal{L}_{ \phi, \mathcal{P} }\curbr{a}(\lambda)$ is exponential with respect to $\lambda$.
\end{lem}

Formally, in the situation of Lemma \ref{Lem20260417}, $\mathcal{L}_{ \phi, \mathcal{P} }\curbr{a}(\lambda) \sim \mbox{const} \times e^{-u(\lambda)}$, where $u(\lambda) \rightarrow \infty$ and $\log(\lambda) = o( u(\lambda) )$, as $\lambda \rightarrow \infty$. Analogous statements can be obtained by means of other integral representations for the kernel $\exp\curbr{-\lambda\phi}$, provided an appropriate bound is available for the associated asymptotic scale. For instance, binomial integrals are considered in Lemma A.2 in \cite{Kaiser04102025}.

\begin{proof}[Proof of Lemma \ref{Lem20260417}]
By assumption, the integral definition of $\mathcal{L}_{\phi, \mathcal{P} }\curbr{a}(\lambda)$ converges absolutely, for all $\lambda \geq 0$. Let $y \in \R$. Then, $| \curbr{\phi(t)}^{-x_0-iy} | = \LandauO( \curbr{\phi(t)}^{-x_0} \exp\curbr{|y|(\frac{\pi}{2}-\varepsilon)} )$, for each $t \in \mathcal{P}$, so that the integral definition of $\mathcal{M}_{\varphi, \mathcal{P} }\curbr{a}(x_0+iy)$ converges absolutely and is $\LandauO( \exp\curbr{|y|(\frac{\pi}{2}-\varepsilon)} )$, as $y \rightarrow \pm \infty$. Besides, it is well-known that $\Gamma(x_0+iy) = \LandauO(\exp\curbr{-|y|\frac{\pi}{2}})$, as $y \rightarrow \pm\infty$ (see, e.g., (5.11.3) in \cite{olver2010nist}). Thus,
\begin{align*}
\mathcal{L}_{\phi, \mathcal{P} }\curbr{a}(\lambda) = \frac{1}{2\pi i} \intl_{x_0-i\infty}^{x_0+i\infty} \lambda^{-z} \Gamma(z) \intl_\mathcal{P} \rrb{\phi(t)}^{-z} a(t) dt dz \hspace{1cm} (\lambda > 0),
\end{align*}
and the double integral converges absolutely, making the asserted result obvious.
\end{proof}

Apparently, the decay of $\mathcal{L}_{ \phi, \mathcal{P} }\curbr{a}(\lambda)$, as $\lambda \rightarrow \infty$, can be arbitrarily slow. Conversely, it may be surprising that it can not happen arbitrarily fast, at least under some widely applicable conditions. For brevity, we write
\begin{align*}
\sigma_- := \infl_{s \in \mathcal{P}} \phi(s), \hspace{1cm} \sigma_+ := \supl_{s \in \mathcal{P}} \phi(s).
\end{align*}

\begin{lem} \label{Lem20260527}
Consider an arbitrary $\mathcal{P} \subseteq \R$ and $\phi : \mathcal{P} \rightarrow [0, \infty)$, with $0 \leq \sigma_- \leq \sigma_+ < \infty$. For $k \in \curbr{1,2}$, let $a_k : \mathcal{P} \rightarrow [0, \infty)$ satisfy $0 < \int_\mathcal{P} a_k(t) dt < \infty$ and define $a(t) := a_1(t)-a_2(t)$, for $t \in \mathcal{P}$. Then, it exist $k_a, K_a > 0$, with
\begin{align*}
k_a e^{-\lambda \sigma_+ } \leq | \mathcal{L}_{\phi, \mathcal{P} }\curbr{a}(\lambda) | \leq K_a e^{-\lambda \sigma_- } \hspace{1cm} (\lambda \rightarrow \infty),
\end{align*}
provided one of the following cases applies:
\begin{enumerate}
\item We can find distinct $j, k \in \curbr{1,2}$, with $\mathcal{L}_{\phi, \mathcal{P} }\curbr{a_j}(\lambda) = o( \mathcal{L}_{\phi, \mathcal{P} }\curbr{a_k}(\lambda) )$, as $\lambda \rightarrow \infty$.
\item $\mathcal{L}_{\phi, \mathcal{P} }\curbr{a_1}(\lambda) \sim c \mathcal{L}_{\phi, \mathcal{P} }\curbr{a_2}(\lambda)$, as $\lambda \rightarrow \infty$, for a fixed $c>0$, with $c \neq 1$.
\end{enumerate}
\end{lem}

The bound fails to hold if $\mathcal{L}_{\phi, \mathcal{P} }\curbr{a_1}(\lambda) \sim \mathcal{L}_{\phi, \mathcal{P} }\curbr{a_2}(\lambda)$, as $\lambda \rightarrow \infty$, because it is then indeed possible to have $\mathcal{L}_{\phi, \mathcal{P} }\curbr{a}(\lambda) = 0$, for all $\lambda \geq 0$. For instance, observe that $\int_0^\pi e^{-\lambda \sin(t)} \cos(t) dt \equiv 0$.

\begin{proof}[Proof of Lemma \ref{Lem20260527}]
Under the present assumptions, the integrals that define $\mathcal{L}_{\phi, \mathcal{P} }\curbr{a_j}(\lambda)$, for $j \in \curbr{1,2}$, and $\mathcal{L}_{\phi, \mathcal{P} }\curbr{a}(\lambda)$ converge absolutely, for all $\lambda \geq 0$. The validity of the upper bound is thus obvious from the triangle inequality. For the proof of the lower bound, for $j \in \curbr{1,2}$, we observe that
\begin{align*}
\mathcal{L}_{\phi, \mathcal{P} }\curbr{a_j}(\lambda) = \intl_\mathcal{P} e^{-\lambda\phi(t)} a_j(t) dt \geq e^{-\lambda \sigma_+ } \intl_\mathcal{P} a_j(t) dt > 0.
\end{align*}
Moreover, for distinct $j, k \in \curbr{1,2}$, we can write
\begin{align*}
\mathcal{L}_{\phi, \mathcal{P} }\curbr{a}(\lambda) = | \mathcal{L}_{\phi, \mathcal{P} }\curbr{a_1}(\lambda) - \mathcal{L}_{\phi, \mathcal{P} }\curbr{a_2}(\lambda) | = \mathcal{L}_{\phi, \mathcal{P} }\curbr{a_j}(\lambda) \abs{ 1 - \frac{ \mathcal{L}_{\phi, \mathcal{P} }\curbr{a_k}(\lambda) }{ \mathcal{L}_{\phi, \mathcal{P} }\curbr{a_j}(\lambda) } }.
\end{align*}
In each of the two assumed cases, the term in absolute value on the right hand side tends to a non-zero limit. The asserted lower inequality is thus a direct consequence of the above estimate for $\mathcal{L}_{\phi, \mathcal{P} }\curbr{a_j}(\lambda)$.
\end{proof}

\end{appendices}

\bibliography{References}

\end{document}